\documentclass[amstex,12pt,reqno]{amsart}
\usepackage{amsmath,amsfonts,amssymb,amsthm,enumerate,hyperref,multicol}
\usepackage{mathrsfs}
\usepackage{graphicx}
\usepackage{breqn}
\usepackage{comment}
\usepackage{float}
\usepackage{subcaption}
\newtheorem{lemma}{Lemma}
\newtheorem{thm}{Theorem}
\newtheorem{prop}{Proposition}
\newtheorem{definition}{Definition}
\newtheorem{example}{Example}
\newtheorem{corollary}{Corollary}
\newtheorem{remark}{Remark}

\newcounter{eqn}
\newcommand{\myeq}[1]{
    \begin{align*}
    #1
    \tag{\theeqn}
    \label{eq:\theeqn}
    \refstepcounter{eqn}
    \end{align*}
}
\numberwithin{equation}{section}

\title[Hermite type sampling Kantorovich operators in Mixed Norm Spaces]
{On Hermite type sampling Kantorovich operators in the settings of mixed norm Spaces}



\author{Puja Sonawane}

\address{Department of Mathematics, Indian Institute of Technology Madras, Chennai-600036, Tamil Nadu, India}

\email{ma22d022@smail.iitm.ac.in}

\author{A. Sathish Kumar}

\address{Department of Mathematics, Indian Institute of Technology Madras, Chennai-600036, Tamil Nadu, India}

\email{sathishkumar@iitm.ac.in, mathsatish9@gmail.com}

\begin{document}

\begin{abstract}
In this paper, we analyze the convergence behavior of Hermite-type sampling Kantorovich operators in the context of mixed norm spaces. We prove certain direct approximation theorems, including the uniform convergence theorem, the Voronovskaja-type asymptotic formula, and an estimate of error in the approximation in terms of the modulus of continuity in mixed norm settings. Next, we estimate the rate of convergence of these sampling Kantorovich operators in terms of the modulus of continuity. In addition, we obtain simultaneous approximation results of these sampling Kantorovich operators, including the uniform approximation, the asymptotic formula, and the approximation error in terms of the modulus of continuity in mixed settings.  
Finally, using cardinal $B$-splines, the implementation of differentiable functions has been shown.

\end{abstract}

\maketitle
\section{Introduction and Preliminaries}\label{section1}

\noindent 
The Shannon Sampling Theorem, a fundamental and important result in Fourier analysis, was independently discovered by Whittaker, Kotelnikov, and Shannon. It states that if a function $f$ is band-limited to $[-\pi w, \pi w],$ $w>0,$ then $f$ can be completely reconstructed by the reconstruction formula 
\begin{eqnarray}\label{eq11}
f(x)=\sum_{k=-\infty}^{\infty}f\bigg(\frac{k}{w}\bigg)\frac{\sin \pi(wx-k)}{\pi(wx-k)},\ \ \ \ (x\in{\mathbb{R}}).
\end{eqnarray}
The above sampling theorem is used in many areas like information theory, signal and image processing, artificial intelligence, etc., mainly because it can reconstruct a signal without losing the actual information. The sampling and reconstruction problems in the setting of uniform and non-uniform sampling have been analyzed by several researchers, see \cite{Ald2, Ald3, but1, Feic, gar1, Heil, jeri, sun} and the references therein. Further, the above sampling series has been analyzed in the direction of approximation theory point of view. In particular, the rate of convergence and $L^p-$ approximation behavior of these sampling series (\ref{eq11}) has been discussed in the series of papers, see \cite{mulima,Lpbut,Fang,Rah} and the references therein.

To approximate the functions that are not necessarily band-limited, Butzer and his school initiated the study of generalized sampling operators. These generalized sampling operators play an important role in signal and image processing, since signals are not necessarily band-limited in nature. So, the approximation behavior of these sampling series is of great importance. For $f:\mathbb{R}\to \mathbb{R},$  $w>0,$ the generalized sampling operator is defined by (see \cite{but6})
\begin{equation}\label{g1}
(S_{w}^{\chi}f)(x)=\sum_{k=-\infty}^{\infty}f\left(\frac{k}{w}\right)\chi(wx-k), \ \ \ \hspace{.5cm} (x\in \mathbb{R}),
\end{equation}
where $\chi:\mathbb{R}\to \mathbb{R}$ is a suitable kernel satisfying certain assumptions. The direct approximation theorems for the above sampling series were discussed in \cite{butss,Stens2}. The approximation of discontinuous signals by generalized sampling series was analyzed in \cite{discont}. The converse approximation theorems of $S_{w}^{\chi}$ were proved in \cite{conver}. The approximation behavior of the multivariate signal by these sampling series has been studied in \cite{plb,plb1}.
The convergence results for the sampling type operators have been studied by several authors in different settings, see \cite{metunc,bmen1,bmen,plb3,cold} and the references therein. Furthermore, to approximate the Lebesgue integrable functions on the real line, the above sampling operators have been modified in \cite{kan1} by replacing the exact sample value $f\left(\frac{k}{w}\right)$ with the average of $f$ in the interval $[k/w,(k+1)/w].$ Such operators are called Kantorovich sampling operators. For $f:\mathbb{R}\to \mathbb{R}, w>0,$ the sampling Kantorovich operator is defined by

\begin{equation}\label{g2}
    (K_{w}^{\chi}f)(x)=\displaystyle\sum_{k=-\infty}^{\infty}\left(w\int_{\frac{k}{w}}^{\frac{k+1}{w}}f(u) du \right)\chi(wx-k), \ \ \ \hspace{.5cm} (x\in \mathbb{R}).
\end{equation}
These operators are useful in signal processing, since more information is usually available in the neighbourhood of a point than at the point itself. In recent years the direct and inverse approximation theorems have been discussed for the Kantorovich sampling operators in different settings; see \cite{ kan3, kan4, kan5, kan6, kan7, kan8,kan2} and the references therein.

Very recently, Corso \cite{cors} considered the following generalized sampling operators to analyze the approximation behavior of a differentiable function. For a $n$ times differentiable function $f:\mathbb{R}\to \mathbb{R},$  $w>0,$ the Hermite type sampling operator is defined by
\begin{equation}\label{g3}
(G_{n,w}f)(x)=\sum_{k=-\infty}^{\infty}\left(\sum_{j=0}^{n}f^{(j)}\left(\frac{k}{w}\right)\left(x-\frac{k}{w}\right)^{j}\right)\chi(wx-k), \ \ \ \hspace{.5cm} (x\in \mathbb{R}). 
\end{equation}
We note that if $n=0,$ the above sampling operator $G_{n,w}$ reduces to the generalized sampling operator $S_{w}^{\chi}.$ Further, we note that if $n>1,$ $w=\frac{1}{\tau}$ and suitable choice of the kernel, the Hermite sampling operators is related to the following sampling operators: For a band-limited and continuous function $f\in L^2(\mathbb{R}),$ we have the following reconstruction formula
\begin{eqnarray}\label{eq12}
f(x)=\sum_{k=-\infty}^{\infty}(f(k\tau)+f^{'}(k\tau)(x-k\tau))\sin c^2\left(\frac{x}{\tau}-k\right),\ \ \ \ \left(\tau=\frac{2}{w}\right)
\end{eqnarray}
for every $x\in\mathbb{R}$. Here, $\sin c$ is given by 
$$\sin c(u)=\begin{cases} 
      \dfrac{\sin \pi u}{\pi u}, & \text{if} \ u\ne 0\\
      1, &  \text{if} \ u=0.

   \end{cases}$$
The above reconstruction formula was proved by Jagerman and Fogel in \cite{jagf}. They have improved the sampling rate as compared to the original Shannon's reconstruction formula. The advantage of the above sampling formula is that one can reconstruct or approximate the functions with the help of a few sample points over an interval of same length. Further to improve the sampling rate Linden and Abramson modified the above formula in \cite{dal} for a band-limited and continuous function $f\in L^2(\mathbb{R}).$ For every $x\in\mathbb{R}$ and for an even greater sampling rate $\tau,$ we have the following reconstruction formula: 
\begin{eqnarray}
f(x)=\sum_{k=-\infty}^{\infty}\left(\sum_{j=0}^{n}\dfrac{1}{j!}f_{j}(k\tau)(x-k\tau)^{j}\right)\sin c^{n+1}\left(\frac{x}{\tau}-k\right),\ \ \ \ \left(\tau=\frac{n+1}{w}\right)
\end{eqnarray}
where
$$f_{j}(x)=\sum_{i=0}^{j}\binom{j}{i}\left(\dfrac{\pi w}{n+1}\right)^{j-i}\Gamma^{j-i}_{n+1}(0)f^{(i)}(x), \ \ \Gamma^{m}_{n+1}(x)=\dfrac{d^{m}}{dx^{m}}\left(\dfrac{1}{\sin c(\frac{x}{\pi})}\right)^{n+1}.$$

Inspired by the above works, we propose and analyze the approximation behavior of Hermite-type sampling Kantorovich operators in the settings of mixed norm. Analyzing the approximations of functions in mixed norm spaces is very important, since in sampling and reconstruction, signals are typically defined in time domains or space fields. However, in practice, some signals are time-varying. This implies that the signals can live in time-space domains simultaneously. A suitable tool for analyzing such signals is mixed norm spaces. This is a major difference from the classical Lebesgue spaces, see \cite{Aben1, Aben2, Fer} and the references therein. Recently, the sampling and reconstruction problems in mixed norm spaces has been studied by several authors, see \cite{mix1, mix2, mix3, mix4, mix5, mixx} and the references therein.

 \subsection{Preliminaries and Notations}
Let  $\mathbb{N}^{d},\mathbb{Z}_{+}^{d}$ and $\mathbb{Z}^{d}$ represent the set of vectors $\mathbf{k}=(k_{1},k_{2},...,k_{d},)$ where $k_{i}$ are positive integers, non-negative integers and integers, respectively, for $ i=1,2,..,d.$  We define $|\mathbf{k}|:=k_{1}+k_{2}+...+k_{d}.$ Furthermore, let $\mathbb{R}^{d}$ denote the $d$ dimensional Euclidean space consisting of all vectors $(x_{1},x_{2},...,x_{d}),$ where $x_{i} \in \mathbb{R}$ for $i=1,2,...,d.$\\
Given $\mathbf{x},\mathbf{y} \in \mathbb{R}^{d}$ and $\alpha \in \mathbb{R},$ the standard operations are given by
$$\mathbf{x}+\mathbf{y}:=(x_{1}+y_{1},x_{2}+y_{2},...,x_{d}+y_{d}),$$
$$\alpha\mathbf{x}:=(\alpha x_{1},\alpha x_{2},...,\alpha x_{d}).$$
Let $ \alpha^{\mathbf{x}}:=(\alpha^{x_{1}},\alpha^{x_{2}},...,\alpha^{x_{d}})$ with $\alpha>0$ and $\mathbf{x}^{\mathbf{y}}:=\prod_{i=1}^{d}x_{i}^{y_{i}},$ when the power is well defined.
If  $\boldsymbol{\alpha},\ \boldsymbol{\beta} \in \mathbb{Z}_{+}^{d},$ we set \\
$$\displaystyle \boldsymbol{\alpha}!=\prod_{i=1}^{d}(\alpha_{i}!), \ \binom{\boldsymbol{\beta}}{\boldsymbol{\alpha}}=\dfrac{\boldsymbol{\beta}!}{\boldsymbol{\alpha}!(\boldsymbol{\beta}-\boldsymbol{\alpha})!}, \ \partial^{\boldsymbol{\alpha}}f=\dfrac{\partial^{|\boldsymbol{\alpha}|}f}{\partial x_{1}^{\alpha_{1}}...\partial x_{d}^{\alpha_{d}}}.$$ \\
By $C_{b}(\mathbb{R}^{d})$ (respectively, $C(\mathbb{R}^{d})$) we denote the space of all continuous and bounded (respectively, uniformly continuous and bounded) functions $f:\mathbb{R}^{d}\to\mathbb{R},$ endowed with the supremum norm $\|f\|_{\infty}:=\displaystyle\sup_{\mathbf{x}\in\mathbb{R}^{d}}|f(\mathbf{x})|.$
Similarly,  $C_{b}(\mathbb{R}\times\mathbb{R}^{d})$ (respectively, $C(\mathbb{R}\times\mathbb{R}^{d})$) we denote the space of all continuous and bounded (respectively, uniformly continuous and bounded) functions $f:\mathbb{R}\times\mathbb{R}^{d}\to \mathbb{R},$ endowed with the supremum norm $\|f\|_{\infty}:=\displaystyle\sup_{\substack{x\in \mathbb{R}\\\mathbf{y}\in\mathbb{R}^{d}}}|f(x,\mathbf{y})|.$
Let  $C^{n}(\mathbb{R}^{d})$ be the space of all $n$-times continuously differentiable functions for which
$\partial^{\boldsymbol{\alpha}}f=\dfrac{\partial^{|\boldsymbol{\alpha}|}f}{\partial x_{1}^{\alpha_{1}}...\partial x_{d}^{\alpha_{d}}}\in C(\mathbb{R}^{d})$ with $|\boldsymbol{\alpha}|=n.$
We define a mixed Lebesgue space as follows. For $1\leq p,q<\infty,$ let $L^{p,q}(\mathbb{R}\times \mathbb{R}^{d})$ denote the space of all complex-valued measurable functions $f$ defined on $\mathbb{R}\times \mathbb{R}^{d}$ such that 
$$\|f\|_{L^{p,q}}:=\left(\int_{\mathbb{R}}\left(\int_{\mathbb{R}^{d}}|f(x,\mathbf{y})|^{q}d\mathbf{y}\right )^{p/q}dx\right)^{1/p}<\infty.$$
Further, let $L^{\infty, \infty}(\mathbb{R} \times \mathbb{R}^{d})$ denote the set of all complex-valued measurable functions on $\mathbb{R}^{d+1}$ such that $\|f\|_{L^{\infty, \infty}} := ess \sup |f| < \infty$.
For $ 1\leq p,q < \infty$, $\ell^{p,q}(\mathbb{Z}\times \mathbb{Z}^{d})$ denotes the space of all complex sequences $c=\big(c(k_{1},k_{2})\big)_{(k_{1}\in \mathbb{Z}, k_{2}\in \mathbb{Z}^{d})}$ such that 
$$\|c\|_{\ell^{p,q}} := \left(\sum_{k_{1}\in \mathbb{Z}}\left(\sum_{k_{2} \in \mathbb{Z}^{d}}|c(k_{1}, k_{2})|^{q}\right)^{p/q}\right)^{1/p} < \infty.$$
We denote by $\ell^{\infty, \infty}(\mathbb{Z} \times \mathbb{Z}^{d}),$ the space of all complex sequences on $\mathbb{Z}^{d+1}$ such that $$\|c\|_{\ell^{\infty, \infty}} := \displaystyle \sup_{k \in \mathbb{Z}^{d+1}}|c(k)| < \infty.$$ 
We note that $L^{p,p}(\mathbb{R} \times \mathbb{R}^{d}) = L^{p}(\mathbb{R}^{d+1})$ and $\ell^{p,p}(\mathbb{Z} \times \mathbb{Z}^{d}) = \ell^{p}(\mathbb{Z}^{d+1})$ for $1 \leq p < \infty.$\vspace{0.5cm} \\ 
Let $\chi:\mathbb{R}\to \mathbb{R}.$ We say that $\chi$ is a kernel of order $n$ if it satisfies the following conditions.
\begin{itemize}
    \item[(i)] 
    $\displaystyle\sum_{k \in \mathbb{Z}} \chi(u-k)=1,$ for every $u \in \mathbb{R}.$
    \item[(ii)] The series 
     $\displaystyle\sum_{k \in \mathbb{Z}} |u-k|^{n} |\chi(u-k)|$ is uniformly convergent in $\mathbb{R}.$
\end{itemize}
For kernel $\chi$ and $\alpha>0,$ we define discrete algebraic moments as
 $$m_{\alpha}(\chi,u):=\displaystyle\sum_{k \in \mathbb{Z}} (k-u)^{\alpha} \chi(u-k), \ u\in \mathbb{R} $$
 and the discrete absolute moment by 
  $$M_{\alpha}(\chi):=\displaystyle\sup_{u\in\mathbb{R}}\sum_{k \in \mathbb{Z}} |u-k|^{\alpha} |\chi(u-k)|.$$

  \begin{remark}\label{rmk1}
     Note that if  $ \chi$ is a kernel of order $ \alpha$, then $\chi$ is kernel of order $\beta,$ i.e. $M_{\beta}(\chi)<\infty $ for every $0\leq\beta\leq\alpha.$
    Further, we have $\displaystyle\lim_{w\to\infty}\sum_{\substack{k \in \mathbb{Z}\\|u-k|\geq w}}|u-k|^{n} |\chi(u-k)|=0$ is uniformly convergent in $\mathbb{R}.$
  \end{remark}
In this paper, we study the approximation behaviour of the Kantorovich-Hermite-type sampling operator in the context of mixed settings. 
\begin{definition}
 Let $f:\mathbb{R}\times\mathbb{R}^{d}\to \mathbb{R}$ be $n$ times differentiable function and $\phi, \psi_{i},\ i=1,...,d,$ are  kernels of order $n.$ For $w>0,$ we define the following Kantorovich-Hermite type sampling operators: 
\begin{eqnarray*}
    K_{n,w}^{\phi,\psi}f(x,\bold{y})&=&\displaystyle\sum_{k\in \mathbb{Z}}\sum_{\mathbf{m} \in\mathbb{Z}^{d}}  \phi(wx-k)\  \prod_{i=1}^{d}\psi_{i}(wy_{i}-m_{i})\\&& \times
\left(\sum_{l+|\bold{j}|\leq n}\dfrac{1}{l! \bold{j}!}w^{d+1}\int_{I_{k}^{w}}\int_{I_{\mathbf{m}}^{w}}\dfrac{\partial^{l+|\bold{j} |}f(u,\bold{v})}{\partial u^{l}\partial \bold{v}^{\bold{j}}}{(x-u)}^{l}{(\bold{y}-\bold{v})^{\bold{j}}}\  du \ d\bold{v}\right),
\end{eqnarray*}
where $x \in \mathbb{R} , \ \bold{y} \in \mathbb{R}^{d},\ I_{k}^{w}=\left[\dfrac{k}{w},\dfrac{k+1}{w}\right] \ \text{and} \ I_{\mathbf{m}}^{w}=\left[\dfrac{m_{1}}{w},\dfrac{m_{1}+1}{w}\right]\times...\times\left[\dfrac{m_{d}}{w},\dfrac{m_{d}+1}{w}\right].$\vspace{0.1cm} We note that if $\phi$ and $\psi_{i},$ $ i=1,...,d,$ are  kernels of order $n,$ and all partial derivatives of $f$ upto order $n$ are bounded, then the sampling operator $K_{n,w}^{\phi,\psi} $ is well-defined. Indeed, we obtain 
 \begin{eqnarray*}
     && |K_{n,w}^{\phi,\psi}f(x,\mathbf{y})| \\
      &&\leq\displaystyle\sum_{k\in \mathbb{Z}}\sum_{\mathbf{m}\in\mathbb{Z}^{d}} |\phi(wx - k)| \prod_{i=1}^{d} |\psi_{i}(wy_{i}-m_{i})| w^{d+1}  \sum_{l+|\bold{j}|\leq n}\dfrac{1}{l! \bold{j}!}  \| \partial^{l+\bold{j} }f\|_{\infty} \\ 
     && \times \int_{ I_{k}^{w}} \int_{ I_{\mathbf{m}}^{w}}  2^{l-1}\left(\left|x-\dfrac{k}{w}\right|^{l}+\left|\dfrac{k}{w}-u\right|^{l}\right) \prod_{i=1}^{d} 2^{j_{i}-1}\left(\left|y_{i}-\dfrac{m_{i}}{w}\right|^{j_{i}}+\left|\dfrac{m_{i}}{w}-v_{i}\right|^{j_{i}}\right) du \ d\bold{v} \\
      &&\leq  \sum_{l+|\bold{j}|\leq n}\dfrac{1}{l! \bold{j}!} \dfrac{2^{l+|\mathbf{j}|-(d+1)}}{w^{l+|\mathbf{j}|}} \| \partial^{l+\bold{j} }f\|_{\infty} (M_{l}(\phi)+M_{0}(\phi)) \prod_{i=1}^{d}  (M_{j_{i}}(\psi_{i})+M_{0}(\psi_{i})),
     \end{eqnarray*}
for all $x \in \mathbb{R} , \ \bold{y} \in \mathbb{R}^{d}.$
\end{definition}
The structure of the paper is organized as follows. The main aim of this paper is to analyze the direct convergence theorems of Hermite-type sampling Kantorovich operators $ K_{n,w}^{\phi,\psi}$ in the context of a mixed norm. In Section \ref{sec2}, we first prove that if $f \in C_{b}^{n}(\mathbb{R} \times \mathbb{R}^{d}),$ then the sampling Kantorovich operators $K_{n,w}^{\phi,\psi}f\in C_{b}^{p}(\mathbb{R} \times \mathbb{R}^{d}).$ Next, we prove the uniform convergence theorem for these sampling operators. Further, we obtain the rate of convergence of $K_{n,w}^{\phi,\psi}$ for $f \in C(\mathbb{R} \times \mathbb{R}^{d}).$ The main tool to estimate the rate of convergence is the modulus of continuity. In this result, we have shown that the error in the approximation can be controlled by the modulus of continuity in the mixed norm. Furthermore, we derive the Voronovskaja-type asymptotic formula for $K_{n,w}^{\phi,\psi}.$ The Section \ref{sec3}, deals with the simultaneous approximation results, including the uniform approximation theorem, the asymptotic formula, and the approximation error in terms of the modulus of continuity in mixed setting for these sampling Kantorovich operators $K_{n,w}^{\phi,\psi}.$ In the last Section, we show the implementation of differentiable functions by these operators using cardinal $B$-spline kernel.

\section{Direct Approximation Results}\label{sec2}
In this section, we discuss certain direct approximation results, including uniform convergence theorem, Voronovskaja-type asymptotic formula and an estimate of error in the approximation in terms of the modulus of continuity in the settings of mixed norm.

\begin{prop}\label{prop1}
   Let $\phi$ and $\psi_{i} \in C_{b}^{p}(\mathbb{R} ), \ i=1,...,d$ be a kernel of order $n$ such that for every $q=1,...,p, \ p\in \mathbb{N},$ the series $$\displaystyle\sum_{k \in \mathbb{Z}} |u-k|^{n} |\phi^{(q)}(u-k)|\ \text{and} \ \sum_{k \in \mathbb{Z}} |u-k|^{n} |\psi_{i}^{(q)}(u-k)|,\ \ i=1,...,d$$ are uniformly convergent in $\mathbb{R}.$
Then for any $f \in C_{b}^{n}(\mathbb{R} \times \mathbb{R}^{d}),$ we have $K_{n,w}^{\phi,\psi}f\in C_{b}^{p}(\mathbb{R} \times \mathbb{R}^{d}).$
\end{prop}
\begin{proof}
    We know that  $\phi, \psi_{i}, \ i=1,...,d $ is a kernel of order $n.$ Then using Remark \ref{rmk1}, the following series
    \myeq{
        \sum_{k \in \mathbb{Z}} |u-k|^{r} |\phi^{(q)}(u-k)| \ and \ \sum_{k \in \mathbb{Z}} |u-k|^{r} |\psi_{i}^{(q)}(u-k)|,\ \ i=1,...,d 
    }
are uniformly convergent in  $\mathbb{R},$ for $r=0,1,...,n$ and $q=0,1,...,p.$
Let us consider the following partial sum $S_{N,w}^{\phi,\psi}$ of the Hermite type Kantorovich sampling operators: 
\begin{eqnarray*}
   S_{N,w}^{\phi,\psi}f(x,\mathbf{y}) &=&\displaystyle\sum\sum_{k+|\mathbf{m}| \leq N}  \phi(wx-k)\  \prod_{i=1}^{d}\psi_{i}(wy_{i}-m_{i})\\&& \times
\left(\sum_{l+|\bold{j}|\leq n}\dfrac{1}{l! \bold{j}!}w^{d+1}\int_{I_{k}^{w}}\int_{I_{\mathbf{m}}^{w}}\dfrac{\partial^{l+|\bold{j} |}f(u,\bold{v})}{\partial u^{l}\partial \bold{v}^{\bold{j}}}{(x-u)}^{l}{(\bold{y}-\bold{v})^{\bold{j}}}\  du \ d\bold{v}\right).
\end{eqnarray*}
Clearly, we have $S_{N,w}^{\phi,\psi}f \in  C_{b}^{p}(\mathbb{R} \times \mathbb{R}^{d}). $
We now show that  $K_{n,w}^{\phi,\psi}f\in C_{b}^{p}(\mathbb{R} \times \mathbb{R}^{d})$ for  $f \in C_{b}^{n}(\mathbb{R} \times \mathbb{R}^{d}).$
Let $\epsilon>0$ be given. We estimate
\begin{eqnarray*}
 &&|K_{n,w}^{\phi,\psi}f(x,\mathbf{y}) - S_{N,w}^{\phi,\psi}f(x,\mathbf{y})|\\
&&\leq \displaystyle\sum\sum_{k+|\mathbf{m}| > N}  |\phi(wx - k)| \prod_{i=1}^{d} |\psi_{i}(wy_{i}-m_{i})| \\
&& \times w^{d+1}  \int_{ I_{k}^{w}} \int_{ I_{\mathbf{m}}^{w}}\left|\sum_{l + |\mathbf{j}| \leq n} \dfrac{1}{l! \mathbf{j}!}  
      \dfrac{\partial^{l + |\mathbf{j}|} f(u, \mathbf{v})}{\partial u^l \partial \mathbf{v}^{\mathbf{j}}} (x - u)^l (\mathbf{y} - \mathbf{v})^{\mathbf{j}} \right| \, du \, d\mathbf{v}.
\end{eqnarray*}

By applying Jensen's inequality, we obtain 
\begin{eqnarray*}
     && |K_{n,w}^{\phi,\psi}f(x,\mathbf{y}) - S_{N,w}^{\phi,\psi}f(x,\mathbf{y})| \\
      &&\leq\displaystyle\sum\sum_{k+|\mathbf{m}| > N} |\phi(wx - k)| \prod_{i=1}^{d} |\psi_{i}(wy_{i}-m_{i})| w^{d+1}  \sum_{l+|\bold{j}|\leq n}\dfrac{1}{l! \bold{j}!}  \| \partial^{l+\bold{j} }f\|_{\infty} \\ 
     && \times \int_{ I_{k}^{w}} \int_{ I_{\mathbf{m}}^{w}}  2^{l-1}\left(\left|x-\dfrac{k}{w}\right|^{l}+\left|\dfrac{k}{w}-u\right|^{l}\right) \prod_{i=1}^{d} 2^{j_{i}-1}\left(\left|y_{i}-\dfrac{m_{i}}{w}\right|^{j_{i}}+\left|\dfrac{m_{i}}{w}-v_{i}\right|^{j_{i}}\right) du \ d\bold{v} 
    \end{eqnarray*}
     \begin{eqnarray*}
      &&\leq \displaystyle\sum\sum_{k+|\mathbf{m}| > N}  |\phi(wx - k)| \prod_{i=1}^{d} |\psi_{i}(wy_{i}-m_{i})| w^{d+1} \sum_{l+|\bold{j}|\leq n}\dfrac{1}{l! \bold{j}!}  \| \partial^{l+\bold{j} }f\|_{\infty}\\ 
       && \times 2^{l-1}\left(\left|x-\dfrac{k}{w}\right|^{l} \dfrac{1}{w}+\dfrac{1}{w^{l+1}}\right) \prod_{i=1}^{d} 2^{j_{i}-1}\left(\left|y_{i}-\dfrac{m_{i}}{w} \right|^{j_{i}}\dfrac{1}{w}+\dfrac{1}{w^{j_{i}+1 }}\right) \\ 
     &&=  \sum_{l+|\bold{j}|\leq n}\dfrac{1}{l! \bold{j}!} \dfrac{2^{l+|\mathbf{j}|-(d+1)}}{w^{l+|\mathbf{j}|}} \| \partial^{l+\bold{j} }f\|_{\infty}\\
      && \times\sum\sum_{k+|\mathbf{m}| > N}  |\phi(wx - k)| \prod_{i=1}^{d} |\psi_{i}(wy_{i}-m_{i})|  \left(\left|wx-k\right|^{l}+1\right) \prod_{i=1}^{d}\left(\left|wy_{i}-m_{i}\right|^{j_{i}}+1\right) \\
      &&\leq \epsilon  \sum_{l+|\bold{j}|\leq n}\dfrac{1}{l! \bold{j}!} \dfrac{2^{l+|\mathbf{j}|-(d+1)}}{w^{l+|\mathbf{j}|}} \| \partial^{l+\bold{j} }f\|_{\infty}.
\end{eqnarray*}
Since $\epsilon>0$ is arbitrary, we have $(S_{N,w}^{\phi,\psi}f)$ converges uniformly to $K_{n,w}^{\phi,\psi}f.$
 Hence, $K_{n,w}^{\phi,\psi}f$ is continuous on $\mathbb{R} \times \mathbb{R}^{d}.$ 
We note that $(S_{N,w}^{\phi,\psi}f)^{'} $ is a linear combination of functions of type $(u-x)^{r},(y_{i}-v_{i})^{r},\phi^{(q)}(wu-k),\psi_{i}^{(q)}(wu-k),i=1,...,d,r=0,1,...,n,q=0,1.$ Using (\ref{eq:1}) and repeating the same approach, we easily obtain $(S_{N,w}^{\phi,\psi}f)^{'} $ is uniformly convergent. Thus, we get $K_{n,w}^{\phi,\psi}f\in  C_{b}^{'}(\mathbb{R} \times \mathbb{R}^{d}). $ 
Iterating the above arguments, we have $K_{n,w}^{\phi,\psi}f\in C_{b}^{p}(\mathbb{R} \times \mathbb{R}^{d}).$
This completes the proof.
 \end{proof}

Next, we prove the following uniform convergence theorem for the Hermite-type Kantorovich sampling operators. 

\begin{thm}\label{thm1}
  Let $\phi$ and $\psi_{i}, \ i=1,...,d$ be a kernel of order $n.$ Suppose that $f \in C_{b}^{n}(\mathbb{R} \times \mathbb{R}^{d}).$ Then, we have $$ \lim_{w \to \infty} \|K_{n,w}^{\phi,\psi}f-f\|_{\infty}=0.$$
\end{thm}
\begin{proof}
By the definition of the Hermite-type sampling Kantorovich operator, we have 
\begin{eqnarray*}  
  && |K_{n,w}^{\phi,\psi}f(x,\mathbf{y}) - f(x,\mathbf{y})|\\
     &&\leq \sum_{k \in \mathbb{Z}} \sum_{\mathbf{m} \in\mathbb{Z}^d} |\phi(wx - k)| \prod_{i=1}^{d} |\psi_{i}(wy_{i}-m_{i})| \\
         &&\times w^{d+1}  \int_{ I_{k}^{w}} \int_{ I_{\mathbf{m}}^{w}}\left|\sum_{l + |\mathbf{j}| \leq n} \dfrac{1}{l! \mathbf{j}!}  
      \dfrac{\partial^{l + |\mathbf{j}|} f(u, \mathbf{v})}{\partial u^l \partial \mathbf{v}^{\mathbf{j}}} (x - u)^l (\mathbf{y} - \mathbf{v})^{\mathbf{j}} - f(x, \mathbf{y}) \right| \, du \, d\mathbf{v}.
   \end{eqnarray*}
By Taylor's formula, we have
\begin{eqnarray*}
    f(x,\bold{y})&&=\displaystyle\sum_{l+|\bold{j}|\leq n-1}\dfrac{1}{l! \bold{j}!}\dfrac{\partial^{l+|\bold{j} |}f(u,\bold{v})}{\partial u^{l}\partial \bold{v}^{\bold{j}}}{(x-u)}^{l}{(\bold{y}-\bold{v})^{\bold{j}}} \\
     && \quad ~~~~~~~~~~~ + \sum_{l+|\bold{j}|= n}\dfrac{1}{l! \bold{j}!}\dfrac{\partial^{l+|\bold{j} |}f(\zeta_{u,x}, \ \zeta_{\bold{v},\bold{y}})}{\partial u^{l}\partial \bold{v}^{\bold{j}}}{(x-u)}^{l}{(\bold{y}-\bold{v})^{\bold{j}}}, 
        \end{eqnarray*}
where $\zeta_{u,x}$ lies between $u$ and $x,$ and  $\zeta_{v_{i},y_{i}} $ lies between $v_{i}$ and $y_{i}$ for $i=1,...,d.$ Thus, we can write 
 \myeq{
   & |K_{n,w}^{\phi,\psi}f(x,\mathbf{y}) - f(x,\mathbf{y})| \\
   &\leq \sum_{k \in \mathbb{Z}} \sum_{\mathbf{m} \in\mathbb{Z}^d} |\phi(wx - k)| \prod_{i=1}^{d} |\psi_{i}(wy_{i}-m_{i})| w^{d+1} \int_{ I_{k}^{w}} \int_{ I_{\mathbf{m}}^{w}} \\ 
    & \times   \left| \sum_{l + |\mathbf{j}| = n} \dfrac{1}{l! \mathbf{j}!} \left( \dfrac{\partial^{l + |\mathbf{j}|} f(u, \mathbf{v})}{\partial u^l \partial \mathbf{v}^{\mathbf{j}}}  -  \dfrac{\partial^{l + |\mathbf{j}|} f(\zeta_{u,x}, \zeta_{\mathbf{v},\mathbf{y}})}{\partial u^l \partial \mathbf{v}^{\mathbf{j}}} \right) (x - u)^l (\mathbf{y} - \mathbf{v})^{\mathbf{j}} \right| \, du \, d\mathbf{v}.
     }
Then, by applying Jensen's inequality, we obtain 
\begin{eqnarray*}
&&|K_{n,w}^{\phi,\psi}f(x,\bold{y})-f(x,\bold{y})|\\
 &&\leq\displaystyle\sum_{k\in \mathbb{Z}}\sum_{\mathbf{m} \in\mathbb{Z}^{d}}  |\phi(wx - k)| \prod_{i=1}^{d} |\psi_{i}(wy_{i}-m_{i})| w^{d+1}  \sum_{l+|\bold{j}|= n}\dfrac{1}{l! \bold{j}!} 2 \| \partial^{l+\bold{j} }f\|_{\infty}  \\
 && \times \int_{ I_{k}^{w}} \int_{ I_{\mathbf{m}}^{w}}  2^{l-1}\left(\left|x-\dfrac{k}{w}\right|^{l}+\left|\dfrac{k}{w}-u\right|^{l}\right) \prod_{i=1}^{d} 2^{j_{i}-1}\left(\left|y_{i}-\dfrac{m_{i}}{w}\right|^{j_{i}}+\left|\dfrac{m_{i}}{w}-v_{i}\right|^{j_{i}}\right) du \ d\bold{v} \\
     &&\leq\displaystyle\sum_{k\in \mathbb{Z}}\sum_{\mathbf{m} \in\mathbb{Z}^{d}}  |\phi(wx - k)| \prod_{i=1}^{d} |\psi_{i}(wy_{i}-m_{i})| w^{d+1} \sum_{l+|\bold{j}|= n}\dfrac{1}{l! \bold{j}!} 2 \| \partial^{l+\bold{j} }f\|_{\infty} \\
      && \times 2^{l-1}\left(\left|x-\dfrac{k}{w}\right|^{l} \dfrac{1}{w}+\dfrac{1}{w^{l+1}}\right) \prod_{i=1}^{d} 2^{j_{i}-1}\left(\left|y_{i}-\dfrac{m_{i}}{w} \right|^{j_{i}}\dfrac{1}{w}+\dfrac{1}{w^{j_{i}+1 }}\right)  \\
       &&\leq \displaystyle \sum_{l+|\bold{j}|= n}\dfrac{1}{l! \bold{j}!}  \dfrac{2^{l+|j|-d}}{w^{l+|j|}} \| \partial^{l+\bold{j} }f\|_{\infty} (M_{l}(\phi)+M_{0}(\phi)) \prod_{i=1}^{d}  (M_{j_{i}}(\psi_{i})+M_{0}(\psi_{i}))\\
 &&= \dfrac{2^{n-d}}{w^{n}}\displaystyle \sum_{l+|\bold{j}|= n}\dfrac{1}{l! \bold{j}!}   \| \partial^{l+\bold{j} }f\|_{\infty} (M_{l}(\phi)+M_{0}(\phi)) \prod_{i=1}^{d}  (M_{j_{i}}(\psi_{i})+M_{0}(\psi_{i})).
\end{eqnarray*}
Since $f \in C_{b}^{n}(\mathbb{R} \times \mathbb{R}^{d})$ and $M_{l}(\phi), M_{j_{i}}(\psi_{i}), \ i=1,...,d$ are finite, then we get $$ \lim_{w \to \infty} \|K_{n,w}^{\phi,\psi}f-f\|_{\infty}=0.$$ 
     Hence, the proof is completed. 
\end{proof}
\subsection{Rate of Convergence}
In this section, we estimate the rate of convergence of the Hermite sampling Kantorovich operators in terms of the modulus of continuity. First, we define the modulus of continuity for $f \in C(\mathbb{R} \times \mathbb{R}^{d}).$  

\begin{definition}
For $f \in C(\mathbb{R} \times \mathbb{R}^{d}),$ the modulus of continuity is defined as follows:
$$ \omega(f,\delta,\gamma_{1},...,\gamma_{d}):= \sup_{\substack{|x-t|< \delta \\ |y_{i}-s_{i}|< \gamma_{i} \\ i=1,...,d}} |f(x,\bold{y})-f(t, \bold{s})|,  \ \ \ (x,\bold{y}), (t,\bold{s}) \in \mathbb{R} \times \mathbb{R}^{d},$$
where $\delta, \gamma_{i}>0,$ for $i=1,...,d.$
\end{definition}
 We note that $\omega(f,\delta,\gamma_{1},...,\gamma_{d})$  satisfies the following properties:
\begin{itemize}
    \item [(i)]
    $\omega(f,\delta,\gamma_{1},...,\gamma_{d}) \to 0$ as $\delta, \gamma_{i} \to 0 , \ \text{for} \  i=1,...,d.$
    \item[(ii)]
    $ |f(x,\bold{y})-f(t, \bold{s})|\leq \omega(f,\delta,\gamma_{1},...,\gamma_{d}) \left(1+\dfrac{|x-t|}{\delta}\right)\displaystyle \prod_{i=1}^{d}\left(1+\dfrac{|y_{i}-s_{i}|}{\gamma_{i}}\right).$
    \end{itemize}
    
\begin{thm} \label{thm2}
   Let $\phi$ and $\psi_{i}, \ i=1,...,d$ be a kernel of order $n.$  
   \begin{itemize}
       \item [(i)] If $f \in C^{n}(\mathbb{R} \times \mathbb{R}^{d}),$ then we have 
       $$\|K_{n,w}^{\phi,\psi}f-f\|_{\infty}= o(w^{-n}) \ \ \text{as}  \ w \to \infty.$$
        \item[(ii)]
       Suppose that $f \in C^{n}(\mathbb{R} \times \mathbb{R}^{d})\ \text{and} \ M_{n+1}(\phi), M_{n+1}(\psi_{i})< \infty , \ i=1,...,d.$ Then, we have  
      \begin{align*}
         & \|K_{n,w}^{\phi,\psi}f-f\|_{\infty}\\
         &\leq\dfrac{2^{n-(d+1)}}{w^{n}} \sum_{l + |\mathbf{j}| = n} \dfrac{1}{l! \mathbf{j}!} \omega\left(\partial^{l+\mathbf{j}}f,\dfrac{1}{w},\dfrac{1}{w},...,\dfrac{1}{w}\right) \left(M_{l+1}(\phi)+2M_{l}(\phi)\right. \\ &
         +\left. M_{1}(\phi)+2 M_{0}(\phi)\right) \prod_{i=1}^{d} \left(M_{j_{i}+1}(\psi_{i})+2M_{j_{i}}(\psi_{i})+
           M_{1}(\psi_{i}) + 2M_{0}(\psi_{i})\right).
      \end{align*}
 \item[(iii)] If $f \in C^{n+1}_{b}(\mathbb{R} \times \mathbb{R}^{d})\ \text{and} \ M_{n+1}(\phi),\ M_{n+1}(\psi_{i})< \infty , \ i=1,...,d,$ then we have  
\begin{eqnarray*}
     &&\|K_{n,w}^{\phi,\psi}f-f\|_{\infty}\\&&\leq
     \dfrac{2^{n-d}}{w^{n+1}}\displaystyle \sum_{l+|\bold{j}|= n+1}\dfrac{1}{l! \bold{j}!}   \| \partial^{l+\bold{j} }f\|_{\infty} (M_{l}(\phi)+M_{0}(\phi)) \prod_{i=1}^{d}  (M_{j_{i}}(\psi_{i})+M_{0}(\psi_{i})).
    \end{eqnarray*}
   \end{itemize}
\end{thm}

\begin{proof}
   First, we prove part (i). From (\ref{eq:2}), we can write
\begin{eqnarray*}
    && |K_{n,w}^{\phi,\psi}f(x,\mathbf{y}) - f(x,\mathbf{y})| \\
   && \leq \sum_{k \in \mathbb{Z}} \sum_{\mathbf{m} \in\mathbb{Z}^d} |\phi(wx - k)| \prod_{i=1}^{d} |\psi_{i}(wy_{i}-m_{i})| w^{d+1}  \\
    &&  \times \int_{ I_{k}^{w}} \int_{ I_{\mathbf{m}}^{w}}   \left| \sum_{l + |\mathbf{j}| = n} \dfrac{1}{l! \mathbf{j}!} \left( \dfrac{\partial^{l + |\mathbf{j}|} f(u, \mathbf{v})}{\partial u^l \partial \mathbf{v}^{\mathbf{j}}}- \dfrac{\partial^{l + |\mathbf{j}|} f(\zeta_{u,x}, \zeta_{\mathbf{v},\mathbf{y}})}{\partial u^l \partial \mathbf{v}^{\mathbf{j}}} \right) (x-u)^l (\mathbf{y}-\mathbf{v})^{\mathbf{j}} \right| \, du \, d\mathbf{v} \\
     &&\leq\displaystyle\sum_{k\in \mathbb{Z}}\sum_{\mathbf{m} \in\mathbb{Z}^{d}}|\phi(wx - k)| \prod_{i=1}^{d} |\psi_{i}(wy_{i}-m_{i})| w^{d+1}  \sum_{l+|\bold{j}|= n}\dfrac{1}{l! \bold{j}!}\\
     &&\times \int_{ I_{k}^{w}} \int_{ I_{\mathbf{m}}^{w}}\left| \dfrac{\partial^{l + |\mathbf{j}|} f(u, \mathbf{v})}{\partial u^l \partial \mathbf{v}^{\mathbf{j}}}-\dfrac{\partial^{l + |\mathbf{j}|} f(\zeta_{u,x}, \zeta_{\mathbf{v},\mathbf{y}})}{\partial u^l \partial \mathbf{v}^{\mathbf{j}}}\right|  2^{l-1}  \left(\left|x-\dfrac{k}{w}\right|^{l}+\left|\dfrac{k}{w}-u\right|^{l}\right) \\
      && \times \prod_{i=1}^{d} 2^{j_{i}-1}\left(\left|y_{i}-\dfrac{m_{i}}{w}\right|^{j_{i}}+\left|\dfrac{m_{i}}{w}-v_{i}\right|^{j_{i}}\right) du \ d\bold{v} \\
    \end{eqnarray*}
\begin{eqnarray*}
 && = \dfrac{2^{n-(d+1)}}{w^{n}}\displaystyle\sum_{k\in \mathbb{Z}}\sum_{\mathbf{m} \in\mathbb{Z}^{d}} |\phi(wx - k)| \prod_{i=1}^{d} |\psi_{i}(wy_{i}-m_{i})| w^{d+1} \sum_{l+|\bold{j}|= n}\dfrac{1}{l! \bold{j}!}\\
     && \times \int_{ I_{k}^{w}} \int_{ I_{\mathbf{m}}^{w}}  \left| \dfrac{\partial^{l + |\mathbf{j}|} f(u, \mathbf{v})}{\partial u^l \partial \mathbf{v}^{\mathbf{j}}}-\dfrac{\partial^{l + |\mathbf{j}|} f(\zeta_{u,x}, \zeta_{\mathbf{v},\mathbf{y}})}{\partial u^l \partial \mathbf{v}^{\mathbf{j}}}\right|  \left(\left|wx-k\right|^{l}+\left|k-wu\right|^{l}\right)   \\
     && \times\prod_{i=1}^{d} \left(\left|wy_{i}-m_{i}\right|^{j_{i}}+\left|m_{i}-wv_{i}\right|^{j_{i}}\right) du \ d\bold{v}. 
\end{eqnarray*}
Since $\partial^{l + \mathbf{j}}f$ is uniform continuous on $\mathbb{R}\times\mathbb{R}^{d},$ for a given $\epsilon >0, \  \exists \ \delta_{l,\bold{j}}>0$ such that
$$|\partial^{l + \mathbf{j}}f(u, \mathbf{v})-\partial^{l + \mathbf{j}}f(x, \mathbf{y})| < \epsilon,$$ whenever $\|(x,\bold{y})-(u,\bold{v})\|_{\infty}\leq \delta_{l,\bold{j}}.$ Thus, we obtain 
\begin{eqnarray*}
     &&w^{n}|K_{n,w}^{\phi,\psi}f(x,\mathbf{y}) - f(x,\mathbf{y})| \\
     && \leq 2^{n-(d+1)}\displaystyle\left(\mathop{\sum\sum}_{\substack{\left\|\left(x,\bold{y}\right)-\left(\dfrac{k}{w},\dfrac{\mathbf{m}}{w}\right)\right\|_{\infty}\leq \dfrac{\delta_{l,\bold{j}}}{2}}}+\mathop{\sum\sum}_{\substack{\left\|\left(x,\bold{y}\right)-\left(\dfrac{k}{w},\dfrac{\mathbf{m}}{w}\right)\right\|_{\infty}> \dfrac{\delta_{l,\bold{j}}}{2}}}\right)\\
     && \times |\phi(wx - k)| \prod_{i=1}^{d} |\psi_{i}(wy_{i}-m_{i})|  w^{d+1} \sum_{l+|\bold{j}|= n}\dfrac{1}{l! \bold{j}!} \int_{ I_{k}^{w}} \int_{ I_{\mathbf{m}}^{w}} \left| \dfrac{\partial^{l + |\mathbf{j}|} f(u, \mathbf{v})}{\partial u^l \partial \mathbf{v}^{\mathbf{j}}}-\dfrac{\partial^{l + |\mathbf{j}|} f(\zeta_{u,x}, \zeta_{\mathbf{v},\mathbf{y}})}{\partial u^l \partial \mathbf{v}^{\mathbf{j}}}\right| \\
      && \times \left(\left|wx-k\right|^{l}+\left|k-wu\right|^{l}\right) \prod_{i=1}^{d} \left(\left|wy_{i}-m_{i}\right|^{j_{i}}+\left|m_{i}-wv_{i}\right|^{j_{i}}\right) du \ d\bold{v}   \\
    &&:= S_{1}+S_{2}.
\end{eqnarray*}

We first estimate $S_{1}.$ We have
\begin{eqnarray*}
     \|(\zeta_{u,x}, \zeta_{\mathbf{v},\mathbf{y}})-(u,\bold{v})\|_{\infty} 
    && \leq\|(x,\bold{y})-(u,\bold{v})\|_{\infty} \\
    &&\leq\left\|\left(x,\bold{y}\right)-\left(\dfrac{k}{w},\dfrac{\mathbf{m}}{w}\right)\right\|_{\infty}+\left\|\left(u,\bold{v}\right)-\left(\dfrac{k}{w},\dfrac{\mathbf{m}}{w}\right)\right\|_{\infty} \\
    &&\leq \dfrac{\delta_{l,\bold{j}}}{2}+\dfrac{1}{w}\leq \delta_{l,\bold{j}}, 
\end{eqnarray*}
for sufficiently large $ w>0,$ with $\dfrac{1}{w}\leq \dfrac{\delta_{l\bold{j}}}{2}$ and for $\left\|\left(x,\bold{y}\right)-\left(\dfrac{k}{w},\dfrac{\mathbf{m}}{w}\right)\right\|_{\infty}\leq\dfrac{\delta_{l\bold{j}}}{2}.$
Thus, we have
$$S_{1}\leq \epsilon \ 2^{n-(d+1)} \displaystyle \sum_{l+|\bold{j}|= n}\dfrac{1}{l! \bold{j}!} (M_{l}(\phi)+M_{0}(\phi)) \prod_{i=1}^{d}  (M_{j_{i}}(\psi_{i})+M_{0}(\psi_{i})).$$ \\
Since $\epsilon>0$ is arbitrary, so we obtain  $S_{1}\to 0$ as $w\to \infty.$ Now $ S_{2}$ is estimated by 
\begin{eqnarray*}
     S_{2}&\leq&2^{n-(d+1)} \displaystyle \sum_{l+|\bold{j}|= n}\dfrac{1}{l! \bold{j}!}  2 \| \partial^{l+\bold{j} }f\|_{\infty} \mathop{\sum\sum}_{\substack{\left\|\left(x,\bold{y}\right)-\left(\dfrac{k}{w},\dfrac{\mathbf{m}}{w}\right)\right\|_{\infty}> \dfrac{\delta_{l,\bold{j}}}{2}}} \\
    & \times& |\phi(wx - k)|   \left(\left|wx-k\right|^{l}+1\right) \prod_{i=1}^{d} |\psi_{i}(wy_{i}-m_{i})|\left(\left|wy_{i}-m_{i}\right|^{j_{i}}+1\right). 
\end{eqnarray*}
Since $ \left\|\left(x,\bold{y}\right)-\left(\dfrac{k}{w},\dfrac{m}{w}\right)\right\|_{\infty}> \dfrac{\delta_{l,\bold{j}}}{2},$ we have
$\left|x-\dfrac{k}{w}\right| >\dfrac{\delta_{l,\bold{j}}}{2}$ or $\left|y_{i}-\dfrac{m_{i}}{w}\right| >\dfrac{\delta_{l,\bold{j}}}{2},$ \ for some \ $i.$ 
This implies that
$ \left|wx-k\right| >\dfrac{w\delta_{l,\bold{j}}}{2}$ or $\left|wy_{i}-m_{i}\right| >\dfrac{w\delta_{l,\bold{j}}}{2},$ \ for some \ $i.$ 
Now, we get 
\begin{eqnarray*}
     S_{2}&\leq&2^{n-(d+1)} \displaystyle \sum_{l+|\bold{j}|= n}\dfrac{1}{l! \bold{j}!}  2 \| \partial^{l+\bold{j} }f\|_{\infty} \left(\sum_{\left|wx-k\right|> \frac{w\delta_{l,\bold{j}}}{2}}\sum...\sum+\sum \sum_{\left|wy_{1}-m_{1}\right|> \frac{w\delta_{l,\bold{j}}}{2}}...\sum\right.\\
    & +&\left. ...+\sum\sum...\sum_{\left|wy_{d}-m_{d}\right|> \frac{w\delta_{l,\bold{j}}}{2}}\right) |\phi(wx - k)|   \left(\left|wx-k\right|^{l}+1\right)\\
    &\times&  \prod_{i=1}^{d} |\psi_{i}(wy_{i}-m_{i})|\left(\left|wy_{i}-m_{i}\right|^{j_{i}}+1\right). 
\end{eqnarray*}
Thus, in view of Remark \ref{rmk1}, we obtain  
$ S_{2} \to 0$ as $w \to \infty.$
Hence, we have
$$\|K_{n,w}^{\phi,\psi}f-f\|_{\infty}= o(w^{-n})\ \ \ \text{as}  \ w \to \infty.$$
Now, we prove part (ii). Again, using (\ref{eq:2}) and the property (ii) of the modulus of continuity, we can write 
\begin{eqnarray*}
     &&|K_{n,w}^{\phi,\psi}f(x,\mathbf{y}) - f(x,\mathbf{y})| \\
   && \leq \sum_{k \in \mathbb{Z}} \sum_{\mathbf{m} \in\mathbb{Z}^d}|\phi(wx - k)| \prod_{i=1}^{d} |\psi_{i}(wy_{i}-m_{i})| w^{d+1} \int_{ I_{k}^{w}} \int_{ I_{\mathbf{m}}^{w}}  \sum_{l + |\mathbf{j}| = n} \dfrac{1}{l! \mathbf{j}!} \omega(\partial^{l+\mathbf{j}}f,\delta,\gamma_{1},...,\gamma_{d}) \\
   && \times   \left(1+\dfrac{|u-\zeta_{u,x}|}{\delta}\right)\displaystyle \prod_{i=1}^{d}\left(1+\dfrac{|v_{i}-\zeta_{{v}_{i},{y}_{i}}|}{\gamma_{i}}\right)|x-u|^l |\mathbf{y}-\mathbf{v}|^{\mathbf{j}}  \, du \, d\mathbf{v} \\
   && \leq \sum_{l + |\mathbf{j}| = n} \dfrac{1}{l! \mathbf{j}!} \omega(\partial^{l+\mathbf{j}}f,\delta,\gamma_{1},...,\gamma_{d})
   \sum_{k \in \mathbb{Z}} \sum_{\mathbf{m} \in\mathbb{Z}^d}|\phi(wx - k)| \prod_{i=1}^{d} |\psi_{i}(wy_{i}-m_{i})| w^{d+1} \\
   && \times \int_{ I_{k}^{w}} \int_{ I_{\mathbf{m}}^{w}}  \left(1+\dfrac{|u-x|}{\delta}\right)\displaystyle \prod_{i=1}^{d}\left(1+\dfrac{|v_{i}-{y}_{i}|}{\gamma_{i}}\right)|x-u|^l |\mathbf{y}-\mathbf{v}|^{\mathbf{j}}  \, du \, d\mathbf{v}. 
\end{eqnarray*}
By using Jensen's inequality and moment conditions, we obtain
\begin{eqnarray*}
    &&|K_{n,w}^{\phi,\psi}f(x,\mathbf{y}) - f(x,\mathbf{y})| \\
   && \leq \dfrac{2^{n-(d+1)}}{w^{n}} \sum_{l + |\mathbf{j}| = n} \dfrac{1}{l! \mathbf{j}!} \omega(\partial^{l+\mathbf{j}}f,\delta,\gamma_{1},...,\gamma_{d}) \left(\dfrac{M_{l+1}(\phi)}{\delta w}+M_{l}(\phi)+
   \dfrac{M_{l}(\phi)}{\delta w}+\dfrac{M_{1}(\phi)}{\delta w}+ M_{0}(\phi)\right. \\
   &&  + \left. \dfrac{M_{0}(\phi)}{\delta w }  \right) \prod_{i=1}^{d} \left(\dfrac{M_{j_{i}+1}(\psi_{i})}{\gamma_{i} w}+M_{j_{i}}(\psi_{i})+
   \dfrac{M_{j_{i}}(\psi_{i})}{\gamma_{i} w}+\dfrac{M_{1}(\psi_{i})}{\gamma_{i} w}+\dfrac{M_{0}(\psi_{i})}{\gamma_{i} w } + M_{0}(\psi_{i})\right).
\end{eqnarray*}
Choosing $\delta=\dfrac{1}{w}$ and $\gamma_{i}=\dfrac{1}{w}, \  i=1,...,d,$ we get 
\begin{eqnarray*}
    &&|K_{n,w}^{\phi,\psi}f(x,\mathbf{y}) - f(x,\mathbf{y})| \\
   && \leq \dfrac{2^{n-(d+1)}}{w^{n}} \sum_{l + |\mathbf{j}| = n} \dfrac{1}{l! \mathbf{j}!} \omega\left(\partial^{l+\mathbf{j}}f,\dfrac{1}{w},\dfrac{1}{w},...,\dfrac{1}{w}\right) \left(M_{l+1}(\phi)+2M_{l}(\phi)+M_{1}(\phi)+2 M_{0}(\phi)\right) \\
   &&  \times \prod_{i=1}^{d} \left(M_{j_{i}+1}(\psi_{i})+2M_{j_{i}}(\psi_{i})+
   M_{1}(\psi_{i}) + 2M_{0}(\psi_{i})\right).
\end{eqnarray*}
Now, we discuss the proof of part (iii). Using Taylor's formula, we can write
\begin{eqnarray*}
     &&|K_{n,w}^{\phi,\psi}f(x,\mathbf{y}) - f(x,\mathbf{y})| \\
   && \leq \sum_{k \in \mathbb{Z}} \sum_{\mathbf{m} \in\mathbb{Z}^d} |\phi(wx - k)| \prod_{i=1}^{d} |\psi_{i}(wy_{i}-m_{i})| w^{d+1}    \\
    &&  \times \int_{ I_{k}^{w}} \int_{ I_{\mathbf{m}}^{w}}\sum_{l + |\mathbf{j}| = n+1} \dfrac{1}{l! \mathbf{j}!}  \left|\dfrac{\partial^{l + |\mathbf{j}|} f(\zeta_{u,x}, \zeta_{\mathbf{v},\mathbf{y}})}{\partial u^l \partial \mathbf{v}^{\mathbf{j}}} \right| |x-u|^l |\mathbf{y}-\mathbf{v}|^{\mathbf{j}}  \, du \, d\mathbf{v}\\
      && \leq \sum_{l + |\mathbf{j}| = n+1} \dfrac{1}{l! \mathbf{j}!} \| \partial^{l+\bold{j} }f\|_{\infty}  \sum_{k \in \mathbb{Z}} \sum_{\mathbf{m} \in\mathbb{Z}^d} |\phi(wx - k)| \prod_{i=1}^{d} |\psi_{i}(wy_{i}-m_{i})|  w^{d+1} \\
     && \times \int_{ I_{k}^{w}} \int_{ I_{\mathbf{m}}^{w}} 2^{l-1}  \left(\left|x-\dfrac{k}{w}\right|^{l}+\left|\dfrac{k}{w}-u\right|^{l}\right) \prod_{i=1}^{d} 2^{j_{i}-1}\left(\left|y_{i}-\dfrac{m_{i}}{w}\right|^{j_{i}}+\left|\dfrac{m_{i}}{w}-v_{i}\right|^{j_{i}}\right) du \ d\bold{v} \\
    && \leq \dfrac{2^{n-d}}{w^{n+1}}\displaystyle \sum_{l+|\bold{j}|= n+1}\dfrac{1}{l! \bold{j}!}   \| \partial^{l+\bold{j} }f\|_{\infty} (M_{l}(\phi)+M_{0}(\phi)) \prod_{i=1}^{d}  (M_{j_{i}}(\psi_{i})+M_{0}(\psi_{i})).
     \end{eqnarray*}
    This completes the proof. 
\end{proof}

In the following theorem, we prove the Voronovskaja-type asymptotic formula for the Kantorovich sampling operators $K_{n,w}^{\phi,\psi}$. 

\begin{thm}\label{thm3}
Let $\phi, \psi_{i},\ i=1,...,d $ be a kernel of order $n+1$ with constant discrete moments $m_{r}(\phi,u), m_{r}(\psi_{i},u), \ r=0,1,...,n+1,\ i=1,...,d$ for all $u\in \mathbb{R}$ and not null. Then for any $f\in C_{b}^{n+1}(\mathbb{R}\times \mathbb{R}^{d}),$ we have 
\begin{eqnarray*}
     &&  \displaystyle\lim_{w\to \infty} w^{n+1}(K_{n,w}^{\phi,\psi}f(x,\mathbf{y}) - f(x,\mathbf{y}))\\
    &&=\displaystyle \sum_{l + |\mathbf{j}| = n+1} \dfrac{(-1)^{n}}{l! \mathbf{j}!} \sum_{c=0}^{l} \sum_{\mathbf{d}=0}^{\mathbf{j}} \binom{l}{c}\binom{\mathbf{j}}{\mathbf{d}}\dfrac{m_{c}(\phi)}{l-c+1}\prod_{i=1}^{d} \dfrac{m_{d_{i}}(\psi_{i})}{j_{i}-{d_{i}+1}} \partial^{l + \mathbf{j}} f(x, \mathbf{y}),
\end{eqnarray*}
for every $(x,\mathbf{y})\in \mathbb{R}\times \mathbb{R}^{d}.$
\end{thm}

\begin{proof}
    Using Taylor's formula, we can write
  \begin{eqnarray*}
        &&K_{n,w}^{\phi,\psi}f(x,\mathbf{y}) - f(x,\mathbf{y}) \\
   && =-\sum_{k \in \mathbb{Z}} \sum_{\mathbf{m} \in\mathbb{Z}^d} \phi(wx - k) \prod_{i=1}^{d} \psi_{i}(wy_{i}-m_{i}) w^{d+1}    \\
    & & \times \int_{ I_{k}^{w}} \int_{ I_{\mathbf{m}}^{w}}\sum_{l + |\mathbf{j}| = n+1} \dfrac{1}{l! \mathbf{j}!}  \partial^{l + \mathbf{j}} f(\zeta_{u,x}, \zeta_{\mathbf{v},\mathbf{y}})  (x-u)^l (\mathbf{y}-\mathbf{v})^{\mathbf{j}}  \, du \, d\mathbf{v}. \\
    &&=\sum_{k \in \mathbb{Z}} \sum_{\mathbf{m} \in\mathbb{Z}^d} \phi(wx - k) \prod_{i=1}^{d} \psi_{i}(wy_{i}-m_{i}) w^{d+1}  \int_{ I_{k}^{w}} \int_{ I_{\mathbf{m}}^{w}}\sum_{l + |\mathbf{j}| = n+1} \dfrac{(-1)^{n}}{l! \mathbf{j}!}  \partial^{l + \mathbf{j}} f(\zeta_{u,x}, \zeta_{\mathbf{v},\mathbf{y}})  \\
    &&  \times \sum_{c=0}^{l} \sum_{\mathbf{d}=0}^{\mathbf{j}} \binom{l}{c}\binom{\mathbf{j}}{\mathbf{d}} \left(\dfrac{k}{w}-x\right)^{c}\left(u-\dfrac{k}{w}\right)^{l-c} \left(\dfrac{\mathbf{m}}{w}-\mathbf{y}\right)^{\mathbf{d}} \left(\mathbf{v}-\dfrac{\mathbf{m}}{w}\right)^{\mathbf{j}-\mathbf{d}} \, du \, d\mathbf{v}. \\
    && = \sum_{l + |\mathbf{j}| = n+1} \dfrac{(-1)^{n}}{l! \mathbf{j}!} \sum_{c=0}^{l} \sum_{\mathbf{d}=0}^{\mathbf{j}} \binom{l}{c}\binom{\mathbf{j}}{\mathbf{d}}\sum_{k \in \mathbb{Z}} \sum_{\mathbf{m} \in\mathbb{Z}^d} \phi(wx - k) \prod_{i=1}^{d} \psi_{i}(wy_{i}-m_{i}) \left(\dfrac{k}{w}-x\right)^{c}\\
    &&\times\left(\dfrac{\mathbf{m}}{w}-\mathbf{y}\right)^{\mathbf{d}}
    w^{d+1}  \int_{ I_{k}^{w}} \int_{ I_{\mathbf{m}}^{w}}\partial^{l + \mathbf{j}} f(\zeta_{u,x}, \zeta_{\mathbf{v},\mathbf{y}}) \left(u-\dfrac{k}{w}\right)^{l-c} \left(\mathbf{v}-\dfrac{\mathbf{m}}{w}\right)^{\mathbf{j}-\mathbf{d}} \, du \, d\mathbf{v}.
  \end{eqnarray*}
We have
\begin{align*}
     & w^{n+1}(K_{n,w}^{\phi,\psi}f(x,\mathbf{y}) - f(x,\mathbf{y}))\\
    &-\displaystyle \sum_{l + |\mathbf{j}| = n+1} \dfrac{(-1)^{n}}{l! \mathbf{j}!} \sum_{c=0}^{l} \sum_{\mathbf{d}=0}^{\mathbf{j}} \binom{l}{c}\binom{\mathbf{j}}{\mathbf{d}}\dfrac{m_{c}(\phi)}{l-c+1}\prod_{i=1}^{d} \dfrac{m_{d_{i}}(\psi_{i})}{j_{i}-{d_{i}+1}} \partial^{l + \mathbf{j}} f(x, \mathbf{y})\\
    &= \sum_{l + |\mathbf{j}| = n+1} \dfrac{(-1)^{n}}{l! \mathbf{j}!} \sum_{c=0}^{l} \sum_{\mathbf{d}=0}^{\mathbf{j}} \binom{l}{c}\binom{\mathbf{j}}{\mathbf{d}}\sum_{k \in \mathbb{Z}} \sum_{\mathbf{m} \in\mathbb{Z}^d} \phi(wx - k) \prod_{i=1}^{d} \psi_{i}(wy_{i}-m_{i})\left(k-wx\right)^{c}\\
    &\times \left(\mathbf{m}-w\mathbf{y}\right)^{\mathbf{d}}
      w^{d+1} \int_{ I_{k}^{w}} \int_{ I_{\mathbf{m}}^{w}}\left(\partial^{l + \mathbf{j}} f(\zeta_{u,x}, \zeta_{\mathbf{v},\mathbf{y}})-\partial^{l + \mathbf{j}} f(x, \mathbf{y})\right)  \left(wu-k\right)^{l-c} \left(w\mathbf{v}-\mathbf{m}\right)^{\mathbf{j}-\mathbf{d}}\, du \, d\mathbf{v}\\
      &:=\sum_{l + |\mathbf{j}| = n+1} \dfrac{(-1)^{n}}{l! \mathbf{j}!} \sum_{c=0}^{l} \sum_{\mathbf{d}=0}^{\mathbf{j}} \binom{l}{c}\binom{\mathbf{j}}{\mathbf{d}} S,
\end{align*}
where
\begin{eqnarray*}
    S&:=&\sum_{k \in \mathbb{Z}} \sum_{\mathbf{m} \in\mathbb{Z}^d} \phi(wx - k) \prod_{i=1}^{d} \psi_{i}(wy_{i}-m_{i})\left(k-wx\right)^{c}
     \left(\mathbf{m}-w\mathbf{y}\right)^{\mathbf{d}}\\
      &\times& w^{d+1} \int_{ I_{k}^{w}} \int_{ I_{\mathbf{m}}^{w}}\left(\partial^{l + \mathbf{j}} f(\zeta_{u,x}, \zeta_{\mathbf{v},\mathbf{y}})-\partial^{l + \mathbf{j}} f(x, \mathbf{y})\right)  \left(wu-k\right)^{l-c} \left(w\mathbf{v}-\mathbf{m}\right)^{\mathbf{j}-\mathbf{d}}\, du \, d\mathbf{v}.
\end{eqnarray*}
Now let $\epsilon>0$ be given. By definition of continuity of 
 $\partial^{l + \mathbf{j}}f$ on $\mathbb{R}\times\mathbb{R}^{d},$ \ $ \exists \ \delta_{l,\bold{j}}>0$ such that
$$|\partial^{l + \mathbf{j}}f(u, \mathbf{v})-\partial^{l + \mathbf{j}}f(x, \mathbf{y})| < \epsilon$$ whenever $\|(x,\bold{y})-(u,\bold{v})\|_{\infty}\leq \delta_{l,\bold{j}}.$ Using this, we can decompose the $S$ as $S_{1}+S_{2},$ where
\begin{eqnarray*}
    S_{1}&:=&\mathop{\sum\sum}_{\substack{\left\|\left(x,\bold{y}\right)-\left(\dfrac{k}{w},\dfrac{\mathbf{m}}{w}\right)\right\|_{\infty}\leq \dfrac{\delta_{l,\bold{j}}}{2}}} \phi(wx - k) \prod_{i=1}^{d} \psi_{i}(wy_{i}-m_{i})w^{d+1}\left(k-wx\right)^{c} \left(\mathbf{m}-w\mathbf{y}\right)^{\mathbf{d}}\\
    &&\times
       w^{d+1}\int_{ I_{k}^{w}} \int_{ I_{\mathbf{m}}^{w}}\left(\partial^{l + \mathbf{j}} f(\zeta_{u,x}, \zeta_{\mathbf{v},\mathbf{y}})-\partial^{l + \mathbf{j}} f(x, \mathbf{y})\right)  \left(wu-k\right)^{l-c} \left(w\mathbf{v}-\mathbf{m}\right)^{\mathbf{j}-\mathbf{d}}\, du \, d\mathbf{v},
\end{eqnarray*}
and
\begin{eqnarray*}
    S_{2}&:=&\mathop{\sum\sum}_{\substack{\left\|\left(x,\bold{y}\right)-\left(\dfrac{k}{w},\dfrac{\mathbf{m}}{w}\right)\right\|_{\infty}> \dfrac{\delta_{l,\bold{j}}}{2}}} \phi(wx - k) \prod_{i=1}^{d} \psi_{i}(wy_{i}-m_{i})w^{d+1}\left(k-wx\right)^{c} \left(\mathbf{m}-w\mathbf{y}\right)^{\mathbf{d}}\\
    &&\times
      w^{d+1} \int_{ I_{k}^{w}} \int_{ I_{\mathbf{m}}^{w}}\left(\partial^{l + \mathbf{j}} f(\zeta_{u,x}, \zeta_{\mathbf{v},\mathbf{y}})-\partial^{l + \mathbf{j}} f(x, \mathbf{y})\right)  \left(wu-k\right)^{l-c} \left(w\mathbf{v}-\mathbf{m}\right)^{\mathbf{j}-\mathbf{d}}\, du \, d\mathbf{v}.
\end{eqnarray*}

Now, using a similar argument as in part (i) of Theorem \ref{thm2}, we can show that $S_{1},  S_{2}\to 0$ as $w\to \infty.$ Hence, we have
\begin{eqnarray*}
    &&  \displaystyle\lim_{w\to \infty} w^{n+1}(K_{n,w}^{\phi,\psi}f(x,\mathbf{y}) - f(x,\mathbf{y}))\\
    &&=\displaystyle \sum_{l + |\mathbf{j}| = n+1} \dfrac{(-1)^{n}}{l! \mathbf{j}!} \sum_{c=0}^{l} \sum_{\mathbf{d}=0}^{\mathbf{j}} \binom{l}{c}\binom{\mathbf{j}}{\mathbf{d}}\dfrac{m_{c}(\phi)}{l-c+1}\prod_{i=1}^{d} \dfrac{m_{d_{i}}(\psi_{i})}{j_{i}-{d_{i}+1}} \partial^{l + \mathbf{j}} f(x, \mathbf{y}).
\end{eqnarray*}
Thus, the proof is completed. 
\end{proof}

   \section{simultaneous approximation} \label{sec3}
In this section, we obtain certain simultaneous approximation results of Hermite-type sampling Kantorovich operators, including the uniform approximation, asymptotic formula, and the approximation error in terms of modulus of continuity in mixed settings.  
\begin{lemma}\label{lemma1}
      Let $\phi, \psi_{i} \in C_{b}^{n}(\mathbb{R} ), \ i=1,...,d$ be a kernel of order $n$ such that for every $l=1,...,n$ the series $$\sum_{k \in \mathbb{Z}} |u-k|^{n} |\phi^{(l)}(u-k)|\ and \ \sum_{k \in \mathbb{Z}} |u-k|^{n} |\psi_{i}^{(l)}(u-k)|,\ \ i=1,...,d$$ are uniformly convergent for $u\in \mathbb{R}.$ For every
 $f \in C^{n}(\mathbb{R} \times \mathbb{R}^{d})$  and  $p+|\mathbf{q}|\leq n,$ we have 
 $$\partial^{p+\mathbf{q}}K_{n,w}^{\phi,\psi}f= \sum_{a=0}^{p} \sum_{\mathbf{b}=0}^{\mathbf{q}} \binom{p}{a}\binom{\mathbf{q}}{\mathbf{b}} w^{a+|\mathbf{b}|}K_{n-(p+|\mathbf{q}|-a-|\mathbf{b}|),w}^{\phi^{(a)},\psi^{(\mathbf{b})}} \partial^{p+\mathbf{q}-a-\mathbf{b}}f.$$
\end{lemma}
\begin{proof}
  We aim to prove the identity   $$\dfrac{\partial^{p}K_{n,w}^{\phi,\psi}f}{\partial x^{p}}= \displaystyle\sum_{a=0}^{p}  \binom{p}{a} w^{a}K_{n-(p-a),w}^{\phi^{(a)},\psi} \dfrac{\partial^{p-a}f}{\partial x^{p-a}}.$$
  For $p=0,$ the result is immediate.
  Now, let us consider the case $p=1.$
  \begin{eqnarray*}
     \dfrac{\partial K_{n,w}^{\phi,\psi}f(x, \mathbf{y})}{\partial x}&= &\displaystyle\sum_{k\in \mathbb{Z}}\sum_{\mathbf{m} \in\mathbb{Z}^{d}} w \ \phi^{'}(wx-k)\  \prod_{i=1}^{d}\psi_{i}(wy_{i}-m_{i})\\
     && \times \left(\sum_{l+|\bold{j}|\leq n}\dfrac{1}{l! \bold{j}!}w^{d+1}\int_{I_{k}^{w}}\int_{I_{\mathbf{m}}^{w}}\dfrac{\partial^{l+|\bold{j} |}f(u,\bold{v})}{\partial u^{l}\partial \bold{v}^{\bold{j}}}{(x-u)}^{l}{(\bold{y}-\bold{v})^{\bold{j}}}\  du \ d\bold{v}\right)\\
     &+&\displaystyle\sum_{k\in \mathbb{Z}}\sum_{\mathbf{m} \in\mathbb{Z}^{d}}  \phi(wx-k)\  \prod_{i=1}^{d}\psi_{i}(wy_{i}-m_{i})\\
     && \times \left(\sum_{l+|\bold{j}|\leq n}\dfrac{1}{l! \bold{j}!}w^{d+1}\int_{I_{k}^{w}}\int_{I_{\mathbf{m}}^{w}}\dfrac{\partial^{l+|\bold{j} |}f(u,\bold{v})}{\partial u^{l}\partial \bold{v}^{\bold{j}}}l {(x-u)}^{l-1}{(\bold{y}-\bold{v})^{\bold{j}}}\  du \ d\bold{v}\right)\\
     &=& w \ K_{n,w}^{\phi^{'},\psi} f(x,\mathbf{y})+ K_{n-1,w}^{\phi,\psi} \frac{\partial f(x,\mathbf{y})}{\partial x}.
\end{eqnarray*}
We now assume that the statement holds for $p-1$ and we prove that the result is true for $p.$\\

\noindent
$\dfrac{\partial^{p}K_{n,w}^{\phi,\psi}f(x, \mathbf{y})}{\partial x^{p}}$
\begin{eqnarray*}
   &=&\dfrac{\partial}{\partial x}\left(\dfrac{\partial^{p-1}K_{n,w}^{\phi,\psi}f(x, \mathbf{y})}{\partial x^{p-1}}\right) \\
   &=& \dfrac{\partial}{\partial x}\left( \displaystyle\sum_{a=0}^{p-1}  \binom{p-1}{a} w^{a}K_{n-(p-1-a),w}^{\phi^{(a)},\psi} \dfrac{\partial^{p-1-a}f(x,\mathbf{y})}{\partial x^{p-1-a}}\right)\\
   &=& \displaystyle\sum_{a=0}^{p-1}  \binom{p-1}{a} w^{a} \left(w\ K_{n-(p-1-a),w}^{\phi^{(a+1)},\psi} \dfrac{\partial^{p-1-a}f(x,\mathbf{y})}{\partial x^{p-1-a}} +K_{n-(p-a),w}^{\phi^{(a)},\psi} \dfrac{\partial^{p-a}f(x,\mathbf{y})}{\partial x^{p-a}}\right)\\
   &=& \displaystyle\sum_{a=0}^{p-1}  \binom{p-1}{a} w^{a} K_{n-(p-a),w}^{\phi^{(a)},\psi} \dfrac{\partial^{p-a}f(x,\mathbf{y})}{\partial x^{p-a}}
   + \displaystyle\sum_{a=1}^{p}  \binom{p-1}{a+1} w^{a} K_{n-(p-a),w}^{\phi^{(a)},\psi} \dfrac{\partial^{p-a}f(x,\mathbf{y})}{\partial x^{p-a}}\\
   &=&\displaystyle\sum_{a=0}^{p}  \binom{p}{a} w^{a} K_{n-(p-a),w}^{\phi^{(a)},\psi} \dfrac{\partial^{p-a}f(x,\mathbf{y})}{\partial x^{p-a}}.
\end{eqnarray*}
Using similar analysis, we can show that
$$\dfrac{\partial^{q_{i}}K_{n,w}^{\phi,\psi}f(x, \mathbf{y})}{\partial y_{i}^{q_{i}}}= \displaystyle\sum_{b_{i}=0}^{q_{i}}  \binom{q_{i}}{b_{i}} w^{b_{i}}K_{n-(q_{i}-b_{i}),w}^{\phi,\psi^{(b_{i})}} \dfrac{\partial^{q_{i}-b_{i}}f(x,\mathbf{y})}{\partial y^{q_{i}-b_{i}}}.$$
Thus, we have 
 $$\partial^{p+\mathbf{q}}K_{n,w}^{\phi,\psi}f= \sum_{a=0}^{p} \sum_{\mathbf{b}=0}^{\mathbf{q}} \binom{p}{a}\binom{\mathbf{q}}{\mathbf{b}} w^{a+|\mathbf{b}|}K_{n-(p+|\mathbf{q}|-a-|\mathbf{b}|),w}^{\phi^{(a)},\psi^{(\mathbf{b})}} \partial^{p+\mathbf{q}-a-\mathbf{b}}f.$$
 This completes the proof.
\end{proof}
   \begin{thm}\label{thm4}
  Let $\phi, \psi_{i} \in C_{b}^{n}(\mathbb{R} ), \ i=1,...,d$ be a kernel of order $n$ such that for every $l=1,...,n,$ the series $$\sum_{k \in \mathbb{Z}} |u-k|^{n} |\phi^{(l)}(u-k)|\ \ and \ \ \sum_{k \in \mathbb{Z}} |u-k|^{n} |\psi_{i}^{(l)}(u-k)|,\ \ i=1,...,d$$ are uniformly convergent for $u\in \mathbb{R}.$ For every
 $f \in C^{n}(\mathbb{R} \times \mathbb{R}^{d})$  and  $p+|\mathbf{q}|\leq n,$ we have 
  $$ \lim_{w \to \infty} \|\partial^{p+\mathbf{q}}K_{n,w}^{\phi,\psi}f-\partial^{p+\mathbf{q}}f\|_{\infty}=0.$$
\end{thm}
\begin{proof}
    Using Proposition \ref{prop1}, we see that $K_{n,w}^{\phi,\psi}f$ is $n$ times differentiable.
    If $ p+|\mathbf{q}|=0$, then the conclusion follows from Theorem \ref{thm1}.
    Let  $1 \leq p+|\mathbf{q}|\leq n.$
   Then, by above Lemma \ref{lemma1}, we have
\myeq{
        \partial^{p+\mathbf{q}}K_{n,w}^{\phi,\psi}f= \sum_{a=0}^{p} \sum_{\mathbf{b}=0}^{\mathbf{q}} \binom{p}{a}\binom{\mathbf{q}}{\mathbf{b}} w^{a+|\mathbf{b}|}K_{n-(p+|\mathbf{q}|-a-|\mathbf{b}|),w}^{\phi^{(a)},\psi^{(\mathbf{b})}} \partial^{p+\mathbf{q}-a-\mathbf{b}}f.
   }
   Then, we get
   \myeq{ 
       \| \partial^{p+\mathbf{q}}K_{n,w}^{\phi,\psi}f- \partial^{p+\mathbf{q}}f\|_{\infty}
       &\leq\|K_{n-(p+|\mathbf{q}|),w}^{\phi,\psi}\partial^{p+\mathbf{q}}f - \partial^{p+\mathbf{q}}f\|_{\infty}\\
       &+\left\|\mathop{\sum_{a=0}^{p} \sum_{\mathbf{b}=0}^{\mathbf{q}}}_{\substack{a + |\mathbf{b}| \geq 1}} \binom{p}{a}\binom{\mathbf{q}}{\mathbf{b}} w^{a+|\mathbf{b}|}K_{n-(p+|\mathbf{q}|-a-|\mathbf{b}|),w}^{\phi^{(a)},\psi^{(\mathbf{b})}} \partial^{p+\mathbf{q}-a-\mathbf{b}}f\right\|_{\infty}.
     }
  
    If $a+|\mathbf{b}|=0$, then we have Kantorovich-Hermite type sampling operator $K_{n-(p+|\mathbf{q}|),w}^{\phi,\psi} $ with kernel $\phi, \psi $ and order $n-(p+|\mathbf{q}|).$ If $ p+|\mathbf{q}|=n,$ then the operator reduces to the usual Knatorovich operator. Then, we have 
    \myeq{
              \lim_{w \to \infty} \|K_{n-(p+|\mathbf{q}|),w}^{\phi,\psi}\partial^{p+\mathbf{q}}f-\partial^{p+\mathbf{q}}f\|_{\infty}=0.
          }
      If $ p+|\mathbf{q}|<n,$ then (\ref{eq:5}) follows from  Theorem \ref{thm1}. 
Now we consider the case $a+|\mathbf{b}|\geq1.$ We have

\begin{eqnarray*}
    &&w^{a+|\mathbf{b}|}K_{n-(p+|\mathbf{q}|-a-|\mathbf{b}|),w}^{\phi^{(a)},\psi^{(\mathbf{b})}} \partial^{p+\mathbf{q}-a-\mathbf{b}}f \\
         &&=w^{a+|\mathbf{b}|} \displaystyle\sum_{k\in \mathbb{Z}}\sum_{\mathbf{m} \in\mathbb{Z}^{d}}  \phi^{(a)}(wx-k)\  \prod_{i=1}^{d}\psi_{i}^{(b_{i})}(wy_{i}-m_{i})\\
         &&\times\left(\sum_{l+|\bold{j}|\leq n-(p+|\mathbf{q}|-a-|\mathbf{b}|)}\dfrac{1}{l! \bold{j}!}w^{d+1}\int_{I_{k}^{w}}\int_{I_{\mathbf{m}}^{w}}\dfrac{\partial^{l+|\bold{j} |+p+|\mathbf{q}|-a-|\mathbf{b}|}f(u,\bold{v})}{\partial u^{l+p-a}\partial \bold{v}^{\bold{j}+\mathbf{q}-\mathbf{b}}}{(x-u)}^{l}{(\bold{y}-\bold{v})^{\bold{j}}}\  du \ d\bold{v}\right)\\
         &&=w^{a+|\mathbf{b}|} \displaystyle\sum_{k\in \mathbb{Z}}\sum_{\mathbf{m} \in\mathbb{Z}^{d}}  \phi^{(a)}(wx-k)\  \prod_{i=1}^{d}\psi_{i}^{(b_{i})}(wy_{i}-m_{i}) w^{d+1}\int_{I_{k}^{w}}\int_{I_{\mathbf{m}}^{w}} \left( \dfrac{\partial^{p+|\mathbf{q}|-a-|\mathbf{b}|}f(x,\bold{y})}{\partial x^{p-a}\partial \bold{y}^{\mathbf{q}-\mathbf{b}}}  \right.\\
         && \left.+ \sum_{l+|\bold{j}|= n-(p+|\mathbf{q}|-a-|\mathbf{b}|)}\dfrac{1}{l! \bold{j}!} 
          \left( \dfrac{\partial^{l+|\bold{j} |+p+|\mathbf{q}|-a-|\mathbf{b}|}f(u,\bold{v})}{\partial u^{l+p-a}\partial \bold{v}^{\bold{j}+\mathbf{q}-\mathbf{b}}} -\dfrac{\partial^{l+|\bold{j} |+p+|\mathbf{q}|-a-|\mathbf{b}|}f(\zeta_{u,x},\zeta_{\bold{v},\bold{y}})}{\partial u^{l+p-a}\partial \bold{v}^{\bold{j}+\mathbf{q}-\mathbf{b}}}\right) \right.  \\
          &&\times  \left. {(x-u)}^{l}{(\bold{y}-\bold{v})^{\bold{j}}} \right)\  du \ d\bold{v}.
\end{eqnarray*}
  Since  $\displaystyle\sum_{k \in \mathbb{Z}} \phi^{(j)}(u-k)=0 \ \text{and} \ \sum_{k \in \mathbb{Z}} \psi^{(j)}_{i}(u-k)=0 , i=1,...,d, \forall u \in \mathbb{R}, \ j\geq1 ,$ we obtain 
   \myeq{
        &w^{a+|\mathbf{b}|}K_{n-(p+\mathbf{q}-a-|\mathbf{b}|),w}^{\phi^{(a)},\psi^{(\mathbf{b})}} \partial^{p+\mathbf{q}-a-\mathbf{b}}f \\ 
        &=w^{a+|\mathbf{b}|} \displaystyle\sum_{k\in \mathbb{Z}}\sum_{\mathbf{m} \in\mathbb{Z}^{d}}  \phi^{(a)}(wx-k)\  \prod_{i=1}^{d}\psi_{i}^{(b_{i})}(wy_{i}-m_{i}) w^{d+1}\int_{I_{k}^{w}}\int_{I_{\mathbf{m}}^{w}} \sum_{l+|\bold{j}|= n-(p+|\mathbf{q}|-a-|\mathbf{b}|)}\dfrac{1}{l! \bold{j}!}  \\
         &  \times
          \left( \dfrac{\partial^{l+|\bold{j} |+p+|\mathbf{q}|-a-|\mathbf{b}|}f(u,\bold{v})}{\partial u^{l+p-a}\partial \bold{v}^{\bold{j}+\mathbf{q}-\mathbf{b}}} -\dfrac{\partial^{l+|\bold{j} |+p+|\mathbf{q}|-a-|\mathbf{b}|}f(\zeta_{u,x},\zeta_{\bold{v},\bold{y}})}{\partial u^{l+p-a}\partial \bold{v}^{\bold{j}+\mathbf{q}-\mathbf{b}}}\right)   {(x-u)}^{l}{(\bold{y}-\bold{v})^{\bold{j}}}\  du \ d\bold{v}.}
   If $p+|\mathbf{q}|<n,$ then we have
  \begin{eqnarray*}
       &&|w^{a+|\mathbf{b}|}K_{n-(p+|\mathbf{q}|-a-|\mathbf{b}|),w}^{\phi^{(a)},\psi^{(\mathbf{b})}} \partial^{p+\mathbf{q}-a-\mathbf{b}}f|\\ 
      && \leq w^{a+|\mathbf{b}|} \displaystyle\sum_{k\in \mathbb{Z}}\sum_{\mathbf{m} \in\mathbb{Z}^{d}}  |\phi^{(a)}(wx-k)|\  \prod_{i=1}^{d}|\psi_{i}^{(b_{i})}(wy_{i}-m_{i})| w^{d+1}\int_{I_{k}^{w}}\int_{I_{\mathbf{m}}^{w}} \sum_{l+|\bold{j}|= n-(p+|\mathbf{q}|-a-|\mathbf{b}|)}\dfrac{1}{l! \bold{j}!}  \\
         &&  \times
          \left| \dfrac{\partial^{l+|\bold{j} |+p+|\mathbf{q}|-a-|\mathbf{b}|}f(u,\bold{v})}{\partial u^{l+p-a}\partial \bold{v}^{\bold{j}+\mathbf{q}-\mathbf{b}}} -\dfrac{\partial^{l+|\bold{j} |+p+|\mathbf{q}|-a-|\mathbf{b}|}f(\zeta_{u,x},\zeta_{\bold{v},\bold{y}})}{\partial u^{l+p-a}\partial \bold{v}^{\bold{j}+\mathbf{q}-\mathbf{b}}}\right|   {|x-u|}^{l}{|\bold{y}-\bold{v}|^{\bold{j}}}\  du \ d\bold{v} \\
          && \leq w^{a+|\mathbf{b}|} \sum_{l+|\bold{j}|= n-(p+|\mathbf{q}|-a-|\mathbf{b}|)}\dfrac{1}{l! \bold{j}!} \dfrac{2^{l+|\bold{j}|-d}}{w^{l+|\bold{j}|}} \|\partial^{l+\bold{j} +p+\mathbf{q}-a-\mathbf{b}}f\|_{\infty}  \\
          && \times \displaystyle\sum_{k\in \mathbb{Z}}\sum_{\mathbf{m} \in\mathbb{Z}^{d}}  |\phi^{(a)}(wx-k)|\ \prod_{i=1}^{d}|\psi_{i}^{(b_{i})}(wy_{i}-m_{i})|  w^{d+1}    \\
          && \times \int_{I_{k}^{w}}\int_{I_{\mathbf{m}}^{w}} \left(\left|wx-k\right|^{l}+\left|k-wu\right|^{l}\right) \prod_{i=1}^{d} \left(\left|wy_{i}-m_{i}\right|^{j_{i}}+\left|m_{i}-wv_{i}\right|^{j_{i}}\right) du \ d\bold{v}
        \end{eqnarray*}
     \begin{eqnarray*}
         &&=\dfrac{2^{ n-(p+|\mathbf{q}|-a-|\mathbf{b}|)-d}}{w^{n-(p+|\mathbf{q}|)}} \sum_{l+|\bold{j}|= n-(p+|\mathbf{q}|-a-|\mathbf{b}|)}\dfrac{1}{l! \bold{j}!} \|\partial^{l+\bold{j} +p+\mathbf{q}-a-\mathbf{b}}f\|_{\infty}  (M_{l}(\phi)+M_{0}(\phi))\\
       && \times \prod_{i=1}^{d}  (M_{j_{i}}(\psi_{i})+M_{0}(\psi_{i})).
     \end{eqnarray*}
  Taking limit as $w\to\infty$ on both sides of the above expression, we get $$|w^{a+|\mathbf{b}|}K_{n-(p+|\mathbf{q}|-a-|\mathbf{b}|),w}^{\phi^{(a)},\psi^{(\mathbf{b})}} \partial^{p+\mathbf{q}-a-\mathbf{b}}f| \to 0 \ \text{as} \ w \to \infty.$$ \\
   Let $p+|\mathbf{q}|=n.$ Then, we obtain
   \begin{eqnarray*}
        &&|w^{a+|\mathbf{b}|}K_{n-(p+|\mathbf{q}|-a-|\mathbf{b}|),w}^{\phi^{(a)},\psi^{(\mathbf{b})}} \partial^{p+\mathbf{q}-a-\mathbf{b}}f|\\ 
      && \leq w^{a+|\mathbf{b}|} \displaystyle\sum_{k\in \mathbb{Z}}\sum_{\mathbf{m} \in\mathbb{Z}^{d}}  |\phi^{(a)}(wx-k)|\  \prod_{i=1}^{d}|\psi_{i}^{(b_{i})}(wy_{i}-m_{i})| w^{d+1}\int_{I_{k}^{w}}\int_{I_{\mathbf{m}}^{w}} \sum_{l+|\bold{j}|= a+|\mathbf{b}|}\dfrac{1}{l! \bold{j}!}  \\
         &&  \times
          \left| \dfrac{\partial^{l+|\bold{j} |+p+|\mathbf{q}|-a-|\mathbf{b}|}f(u,\bold{v})}{\partial u^{l+p-a}\partial \bold{v}^{\bold{j}+\mathbf{q}-\mathbf{b}}} -\dfrac{\partial^{l+|\bold{j} |+p+|\mathbf{q}|-a-|\mathbf{b}|}f(\zeta_{u,x},\zeta_{\bold{v},\bold{y}})}{\partial u^{l+p-a}\partial \bold{v}^{\bold{j}+\mathbf{q}-\mathbf{b}}}\right|   {|x-u|}^{l}{|\bold{y}-\bold{v}|^{\bold{j}}}\  du \ d\bold{v}. 
   \end{eqnarray*}
       Using the similar analysis as in Theorem \ref{thm3}, we obtain
   \myeq{|w^{a+|\mathbf{b}|}K_{n-(p+|\mathbf{q}|-a-|\mathbf{b}|),w}^{\phi^{(a)},\psi^{(\mathbf{b})}} \partial^{p+\mathbf{q}-a-\mathbf{b}}f| \to 0 \ \text{as} \ w \to \infty.} 
         
   Now, combining the estimates (\ref{eq:4}), (\ref{eq:5}) and (\ref{eq:7}), we get the required result.
   \end{proof}
   
   \begin{thm}\label{thm5}
Let $\phi,  \psi_{i} \in C_{b}^{n}(\mathbb{R} ), \ i=1,...,d$ be a kernel of order $n$ such that for every $l=1,...,n,$ the series $$\displaystyle\sum_{k \in \mathbb{Z}} |u-k|^{n} |\phi^{(l)}(u-k)|\ \ and \ \ \sum_{k \in \mathbb{Z}} |u-k|^{n} |\psi_{i}^{(l)}(u-k)|,\ \ i=1,...,d$$ are uniformly convergent for $u\in \mathbb{R}.$ Further, let
 $f \in C^{n}(\mathbb{R} \times \mathbb{R}^{d}),$ $p+|\mathbf{q}|\leq n$ 
    and \ $M_{n+1}(\phi),  M_{n+1}(\psi_{i})< \infty , \ i=1,...,d.$ Then, we have\\

    \noindent
    $ \|\partial^{p+\mathbf{q}}K_{n,w}^{\phi,\psi}f-\partial^{p+\mathbf{q}}f\|_{\infty}$
    \begin{eqnarray*}
    &&\leq \dfrac{1}{w^{n-(p+|\mathbf{q}|)}}\sum_{a=0}^{p} \sum_{\mathbf{b}=0}^{\mathbf{q}} \binom{p}{a}\binom{\mathbf{q}}{\mathbf{b}} 2^{n-(p+|\mathbf{q}|-a-|\mathbf{b}|)-(d+1)} \sum_{l+|\bold{j}|= n-(p+|\mathbf{q}|-a-|\mathbf{b}|)}\dfrac{1}{l! \bold{j}!}   \\
         && \times \omega\left(\partial^{l+\mathbf{j}+p+\mathbf{q}-a-\mathbf{b}}f,\dfrac{1}{w},\dfrac{1}{w},...,\dfrac{1}{w}\right)\left(M_{l+1}(\phi^{(a)})  +2M_{l}(\phi^{(a)})+M_{1}(\phi^{(a)})+ 
  2 M_{0}(\phi^{(a)})\right)\\
   && \times\prod_{i=1}^{d} \left(M_{j_{i}+1}(\psi_{i}^{(b_{i})})+2M_{j_{i}}(\psi_{i}^{(b_{i})})
   +M_{1}(\psi_{i}^{(b_{i})})+2 M_{0}(\psi_{i}^{(b_{i})})\right).
    \end{eqnarray*}
\end{thm}

\begin{proof}

    Taking $a+|\mathbf{b}|=0$ in (\ref{eq:3}) and using  part (ii) of Theorem \ref{thm2}, we can write 
    \myeq{ 
        &\|K_{n-(p+|\mathbf{q}|),w}^{\phi,\psi}\partial^{p+\mathbf{q}}f-\partial^{p+\mathbf{q}}f\|_{\infty}\\
        &\leq\dfrac{2^{n-(p+|\mathbf{q}|)-(d+1)}}{w^{n-(p+|\mathbf{q}|)}} \sum_{l + |\mathbf{j}| = n-(p+|\mathbf{q}|)} \dfrac{1}{l! \mathbf{j}!} \omega(\partial^{l+\mathbf{j}+p+\mathbf{q}}f,\delta,\gamma_{1},...,\gamma_{d})  \\ &
         \times \left(M_{l+1}(\phi)+2M_{l}(\phi)+M_{1}(\phi)+2 M_{0}(\phi)\right) \prod_{i=1}^{d} \left(M_{j_{i}+1}(\psi_{i})+2M_{j_{i}}(\psi_{i})+
           M_{1}(\psi_{i}) + 2M_{0}(\psi_{i})\right).
      }
    Consider $a+|\mathbf{b}|\geq1$ in (\ref{eq:3}), using (\ref{eq:6}) and property (ii) of modulus of continuity, we obtain
  \begin{eqnarray*}
        &&|w^{a+|\mathbf{b}|}K_{n-(p+|\mathbf{q}|-a-|\mathbf{b}|),w}^{\phi^{(a)},\psi^{(\mathbf{b})}} \partial^{p+\mathbf{q}-a-\mathbf{b}}f|\\ 
      && \leq w^{a+|\mathbf{b}|} \displaystyle\sum_{k\in \mathbb{Z}}\sum_{\mathbf{m} \in\mathbb{Z}^{d}}  |\phi^{(a)}(wx-k)|\  \prod_{i=1}^{d}|\psi_{i}^{(b_{i})}(wy_{i}-m_{i})| w^{d+1}\int_{I_{k}^{w}}\int_{I_{\mathbf{m}}^{w}} \sum_{l+|\bold{j}|= n-(p+|\mathbf{q}|-a-|\mathbf{b}|)}\dfrac{1}{l! \bold{j}!}  \\
         &&  \times
          \left| \dfrac{\partial^{l+|\bold{j} |+p+|\mathbf{q}|-a-|\mathbf{b}|}f(u,\bold{v})}{\partial u^{l+p-a}\partial \bold{v}^{\bold{j}+\mathbf{q}-\mathbf{b}}} -\dfrac{\partial^{l+|\bold{j} |+p+|\mathbf{q}|-a-|\mathbf{b}|}f(\zeta_{u,x},\zeta_{\bold{v},\bold{y}})}{\partial u^{l+p-a}\partial \bold{v}^{\bold{j}+\mathbf{q}-\mathbf{b}}}\right|   {|x-u|}^{l}{|\bold{y}-\bold{v}|^{\bold{j}}}\  du \ d\bold{v}\\
          &&\leq  w^{a+|\mathbf{b}|} \displaystyle\sum_{k\in \mathbb{Z}}\sum_{\mathbf{m} \in\mathbb{Z}^{d}}  |\phi^{(a)}(wx-k)|\  \prod_{i=1}^{d}|\psi_{i}^{(b_{i})}(wy_{i}-m_{i})| \sum_{l+|\bold{j}|= n-(p+|\mathbf{q}|-a-|\mathbf{b}|)}\dfrac{1}{l! \bold{j}!}  \\
     && \times \omega(\partial^{l+\mathbf{j}+p+\mathbf{q}-a-\mathbf{b}}f,\delta,\gamma_{1},...,\gamma_{d}) w^{d+1} \int_{I_{k}^{w}}\int_{I_{\mathbf{m}}^{w}}
     \left(1+\dfrac{|u-\zeta_{u,x}|}{\delta}\right)\displaystyle \prod_{i=1}^{d}\left(1+\dfrac{|v_{i}-\zeta_{{v}_{i},{y}_{i}}|}{\gamma_{i}}\right)\\
     && \times  {|x-u|}^{l}{|\bold{y}-\bold{v}|^{\bold{j}}}\  du \ d\bold{v}.
  \end{eqnarray*}
     Using the similar analysis from part (ii) of Theorem \ref{thm2} and in view of Jensen's inequality, moment condition, we can write
    \begin{eqnarray*}
         &&|w^{a+|\mathbf{b}|}K_{n-(p+|\mathbf{q}|-a-|\mathbf{b}|),w}^{\phi^{(a)},\psi^{(\mathbf{b})}} \partial^{p+\mathbf{q}-a-\mathbf{b}}f|\\
     && \leq\dfrac{2^{n-(p+|\mathbf{q}|-a-|\mathbf{b}|)-(d+1)}}{ w^{n-(p+|\mathbf{q}|)}} \sum_{l+|\bold{j}|= n-(p+|\mathbf{q}|-a-|\mathbf{b}|)}\dfrac{1}{l! \bold{j}!} \omega(\partial^{l+\mathbf{j}+p+\mathbf{q}-a-\mathbf{b}}f,\delta,\gamma_{1},...,\gamma_{d}) \left(\dfrac{M_{l+1}(\phi^{(a)})}{\delta w}  \right. \\
         && \left.+M_{l}(\phi^{(a)}) +
   \dfrac{M_{l}(\phi^{(a)})}{\delta w}+\dfrac{M_{1}(\phi^{(a)})}{\delta w} 
  \dfrac{M_{0}(\phi^{(a)})}{\delta w } + M_{0}(\phi^{(a)})\right) \prod_{i=1}^{d} \left(\dfrac{M_{j_{i}+1}(\psi_{i}^{(b_{i})})}{\gamma_{i} w}+M_{j_{i}}(\psi_{i}^{(b_{i})})\right.\\
   && \left. +
   \dfrac{M_{j_{i}}(\psi_{i}^{(b_{i})})}{\gamma_{i} w}+\dfrac{M_{1}(\psi_{i}^{(b_{i})})}{\gamma_{i} w}+\dfrac{M_{0}(\psi_{i}^{(b_{i})})}{\gamma_{i} w } + M_{0}(\psi_{i}^{(b_{i})})\right).
    \end{eqnarray*}
     Choosing $\delta=\dfrac{1}{w}$ and $\gamma_{i}=\dfrac{1}{w}, \ i=1,...,d,$ we obtain 
    \myeq{ 
     &|w^{a+|\mathbf{b}|}K_{n-(p+|\mathbf{q}|-a-|\mathbf{b}|),w}^{\phi^{(a)},\psi^{(\mathbf{b})}} \partial^{p+\mathbf{q}-a-\mathbf{b}}f|\\
     & \leq\dfrac{2^{n-(p+|\mathbf{q}|-a-|\mathbf{b}|)-(d+1)}}{ w^{n-(p+|\mathbf{q}|)}} \sum_{l+|\bold{j}|= n-(p+|\mathbf{q}|-a-|\mathbf{b}|)}\dfrac{1}{l! \bold{j}!} \omega\left(\partial^{l+\mathbf{j}+p+\mathbf{q}-a-\mathbf{b}}f,\dfrac{1}{w},\dfrac{1}{w},...,\dfrac{1}{w}\right) \\ 
         &\times \left(M_{l+1}(\phi^{(a)})  +2M_{l}(\phi^{(a)})+M_{1}(\phi^{(a)})+ 
       2 M_{0}(\phi^{(a)})\right) \prod_{i=1}^{d} \left(M_{j_{i}+1}(\psi_{i}^{(b_{i})})+2M_{j_{i}}(\psi_{i}^{(b_{i})})\right.\\
       & \left.
      +M_{1}(\psi_{i}^{(b_{i})})+2 M_{0}(\psi_{i}^{(b_{i})})\right).
       }
     Thus, using the estimate (\ref{eq:3}), (\ref{eq:8}) and (\ref{eq:9}), we get
     \begin{eqnarray*}
         &&\|\partial^{p+\mathbf{q}}K_{n,w}^{\phi,\psi}f-\partial^{p+\mathbf{q}}f\|_{\infty}\\
         &&\leq \sum_{a=0}^{p} \sum_{\mathbf{b}=0}^{\mathbf{q}} \binom{p}{a}\binom{\mathbf{q}}{\mathbf{b}} \dfrac{2^{n-(p+|\mathbf{q}|-a-|\mathbf{b}|)-(d+1)}}{ w^{n-(p+|\mathbf{q}|)}} \sum_{l+|\bold{j}|= n-(p+|\mathbf{q}|-a-|\mathbf{b}|)}\dfrac{1}{l! \bold{j}!}   \\
         && \times \omega\left(\partial^{l+\mathbf{j}+p+\mathbf{q}-a-\mathbf{b}}f,\dfrac{1}{w},\dfrac{1}{w},...,\dfrac{1}{w}\right)\left(M_{l+1}(\phi^{(a)})  +2M_{l}(\phi^{(a)})+M_{1}(\phi^{(a)})+ 
  2 M_{0}(\phi^{(a)})\right)\\
   && \times \prod_{i=1}^{d} \left(M_{j_{i}+1}(\psi_{i}^{(b_{i})})+2M_{j_{i}}(\psi_{i}^{(b_{i})})
   +M_{1}(\psi_{i}^{(b_{i})})+2 M_{0}(\psi_{i}^{(b_{i})})\right)\\
    &&\leq \dfrac{1}{w^{n-(p+|\mathbf{q}|)}}\sum_{a=0}^{p} \sum_{\mathbf{b}=0}^{\mathbf{q}} \binom{p}{a}\binom{\mathbf{q}}{\mathbf{b}} 2^{n-(p+|\mathbf{q}|-a-|\mathbf{b}|)-(d+1)} \sum_{l+|\bold{j}|= n-(p+|\mathbf{q}|-a-|\mathbf{b}|)}\dfrac{1}{l! \bold{j}!}   \\
         && \times \omega\left(\partial^{l+\mathbf{j}+p+\mathbf{q}-a-\mathbf{b}}f,\dfrac{1}{w},\dfrac{1}{w},...,\dfrac{1}{w}\right)\left(M_{l+1}(\phi^{(a)})  +2M_{l}(\phi^{(a)})+M_{1}(\phi^{(a)})+ 
  2 M_{0}(\phi^{(a)})\right)\\
   && \times\prod_{i=1}^{d} \left(M_{j_{i}+1}(\psi_{i}^{(b_{i})})+2M_{j_{i}}(\psi_{i}^{(b_{i})})
   +M_{1}(\psi_{i}^{(b_{i})})+2 M_{0}(\psi_{i}^{(b_{i})})\right).
     \end{eqnarray*}
  Hence, we get the desired estimate. 
\end{proof}

 \begin{thm}\label{thm6}
 Let $\phi, \psi_{i} \in C_{b}^{n}(\mathbb{R} ), i=1,...,d $ be a kernel of order $n+1$ with constant discrete moment such that for every $l=1,...,n,$ the series $$\displaystyle\sum_{k \in \mathbb{Z}} |u-k|^{n+1} |\phi^{(l)}(u-k)|\ and \ \sum_{k \in \mathbb{Z}} |u-k|^{n+1} |\psi_{i}^{(l)}(u-k)|,\ \ i=1,...,d$$ are uniformly convergent for $u\in \mathbb{R}.$ Suppose that $p+|\mathbf{q}|\leq n$ and $f\in C_{b}^{n+1}(\mathbb{R}\times \mathbb{R}^{d}).$ Then, we have\newpage
    \begin{eqnarray*}
          &&  \displaystyle\lim_{w\to \infty} w^{n-(p+|\mathbf{q}|)+1} (\partial^{p+\mathbf{q}}K_{n,w}^{\phi,\psi}f(x,\mathbf{y})-\partial^{p+\mathbf{q}}f(x,\mathbf{y}))\\
    &&=\sum_{a=0}^{p} \sum_{\mathbf{b}=0}^{\mathbf{q}} \binom{p}{a}\binom{\mathbf{q}}{\mathbf{b}}\displaystyle \sum_{l + |\mathbf{j}| = n-(p+|\mathbf{q}|-a-|\mathbf{b}|)+1} \dfrac{(-1)^{n-(p+|\mathbf{q}|-a-|\mathbf{b}|)}}{l! \mathbf{j}!} \sum_{c=0}^{l} \sum_{\mathbf{d}=0}^{\mathbf{j}} \binom{l}{c}\binom{\mathbf{j}}{\mathbf{d}}\\
    &&\times\dfrac{m_{c}(\phi^{(a)})}{l-c+1}\prod_{i=1}^{d} \dfrac{m_{d_{i}}(\psi_{i}^{(b_{i})})}{j_{i}-{d_{i}+1}} \partial^{l + \mathbf{j}+p+\mathbf{q}-a-\mathbf{b}} f(x, \mathbf{y}),
    \end{eqnarray*}  
     for any $(x,\mathbf{y})\in \mathbb{R}\times \mathbb{R}^{d}.$
 \end{thm}  
 
 \begin{proof}
     Taking $a+|\mathbf{b}|=0$ in (\ref{eq:3}) and using Theorem \ref{thm3}, we can write  
     \myeq{ 
          &  \displaystyle\lim_{w\to \infty} w^{n-(p+|\mathbf{q}|)+1} (K_{n-(p+|\mathbf{q}|),w}^{\phi,\psi}\partial^{p+\mathbf{q}}f(x,\mathbf{y})-\partial^{p+\mathbf{q}}f(x,\mathbf{y}))\\
    &=\displaystyle \sum_{l + |\mathbf{j}| = n-(p+|\mathbf{q}|)+1} \dfrac{(-1)^{n-(p+|\mathbf{q}|)}}{l! \mathbf{j}!} \sum_{c=0}^{l} \sum_{\mathbf{d}=0}^{\mathbf{j}} \binom{l}{c}\binom{\mathbf{j}}{\mathbf{d}}\dfrac{m_{c}(\phi)}{l-c+1}\prod_{i=1}^{d} \dfrac{m_{d_{i}}(\psi_{i})}{j_{i}-{d_{i}+1}} \partial^{l + \mathbf{j}+p+\mathbf{q}} f(x, \mathbf{y}).
       }
     Taking $a+|\mathbf{b}|\geq1$ in (\ref{eq:3}) and using Taylor's formula, we can write
     \begin{eqnarray*}
          &&w^{a+|\mathbf{b}|}K_{n-(p+|\mathbf{q}|-a-|\mathbf{b}|),w}^{\phi^{(a)},\psi^{(\mathbf{b})}} \partial^{p+\mathbf{q}-a-\mathbf{b}}f(x,\mathbf{y}) \\
         &&=w^{a+|\mathbf{b}|} \displaystyle\sum_{k\in \mathbb{Z}}\sum_{\mathbf{m} \in\mathbb{Z}^{d}}  \phi^{(a)}(wx-k)\  \prod_{i=1}^{d}\psi_{i}^{(b_{i})}(wy_{i}-m_{i})\\
         &&\times\left(\sum_{l+|\bold{j}|\leq n-(p+|\mathbf{q}|-a-|\mathbf{b}|)}\dfrac{1}{l! \bold{j}!}w^{d+1}\int_{I_{k}^{w}}\int_{I_{\mathbf{m}}^{w}}\dfrac{\partial^{l+|\bold{j} |+p+|\mathbf{q}|-a-|\mathbf{b}|}f(u,\bold{v})}{\partial u^{l+p-a}\partial \bold{v}^{\bold{j}+\mathbf{q}-\mathbf{b}}}{(x-u)}^{l}{(\bold{y}-\bold{v})^{\bold{j}}}\  du \ d\bold{v}\right)\\
          &&=w^{a+|\mathbf{b}|} \displaystyle\sum_{k\in \mathbb{Z}}\sum_{\mathbf{m} \in\mathbb{Z}^{d}}  \phi^{(a)}(wx-k)\  \prod_{i=1}^{d}\psi_{i}^{(b_{i})}(wy_{i}-m_{i}) w^{d+1}\int_{I_{k}^{w}}\int_{I_{\mathbf{m}}^{w}} \left( \dfrac{\partial^{p+|\mathbf{q}|-a-|\mathbf{b}|}f(x,\bold{y})}{\partial x^{p-a}\partial \bold{y}^{\mathbf{q}-\mathbf{b}}}  \right.\\
         && \left.- \sum_{l+|\bold{j}|= n-(p+|\mathbf{q}|-a-|\mathbf{b}|)+1}\dfrac{1}{l! \bold{j}!} 
          \dfrac{\partial^{l+|\bold{j} |+p+|\mathbf{q}|-a-|\mathbf{b}|}f(\zeta_{u,x},\zeta_{\bold{v},\bold{y}})}{\partial u^{l+p-a}\partial \bold{v}^{\bold{j}+\mathbf{q}-\mathbf{b}}} {(x-u)}^{l}{(\bold{y}-\bold{v})^{\bold{j}}} \right)\  du \ d\bold{v}\\
           &&=w^{a+|\mathbf{b}|} \displaystyle\sum_{k\in \mathbb{Z}}\sum_{\mathbf{m} \in\mathbb{Z}^{d}}  \phi^{(a)}(wx-k)\  \prod_{i=1}^{d}\psi_{i}^{(b_{i})}(wy_{i}-m_{i}) \sum_{l+|\bold{j}|= n-(p+|\mathbf{q}|-a-|\mathbf{b}|)+1} \\
           &&\times\dfrac{(-1)^{n-(p+|\mathbf{q}|-a-|\mathbf{b}|)}}{l! \bold{j}!}  w^{d+1}\int_{I_{k}^{w}}\int_{I_{\mathbf{m}}^{w}}
          \dfrac{\partial^{l+|\bold{j} |+p+|\mathbf{q}|-a-|\mathbf{b}|}f(\zeta_{u,x},\zeta_{\bold{v},\bold{y}})}{\partial u^{l+p-a}\partial \bold{v}^{\bold{j}+\mathbf{q}-\mathbf{b}}} \\
        &&\times  \sum_{c=0}^{l} \sum_{\mathbf{d}=0}^{\mathbf{j}} \binom{l}{c}\binom{\mathbf{j}}{\mathbf{d}} \left(u-\dfrac{k}{w}\right)^{l-c} \left(\dfrac{k}{w}-x\right)^{c} \left(\mathbf{v}-\dfrac{\mathbf{m}}{w}\right)^{\mathbf{j}-\mathbf{d}} \left(\dfrac{\mathbf{m}}{w}-\mathbf{y}\right)^{\mathbf{d}} \  du \ d\bold{v}
     \end{eqnarray*}
      \begin{eqnarray*}
      &&=w^{a+|\mathbf{b}|} \displaystyle\sum_{k\in \mathbb{Z}}\sum_{\mathbf{m} \in\mathbb{Z}^{d}}  \phi^{(a)}(wx-k)\  \prod_{i=1}^{d}\psi_{i}^{(b_{i})}(wy_{i}-m_{i}) \sum_{l+|\bold{j}|= n-(p+|\mathbf{q}|-a-|\mathbf{b}|)+1}  \\
          &&\times \dfrac{(-1)^{n-(p+|\mathbf{q}|-a-|\mathbf{b}|)}}{l! \bold{j}!} \sum_{c=0}^{l} \sum_{\mathbf{d}=0}^{\mathbf{j}} \binom{l}{c}\binom{\mathbf{j}}{\mathbf{d}}
          \left(\dfrac{k}{w}-x\right)^{c}\left(\dfrac{\mathbf{m}}{w}-\mathbf{y}\right)^{\mathbf{d}}\\
          && \times w^{d+1}\int_{I_{k}^{w}}\int_{I_{\mathbf{m}}^{w}}
          \dfrac{\partial^{l+|\bold{j} |+p+|\mathbf{q}|-a-|\mathbf{b}|}f(\zeta_{u,x},\zeta_{\bold{v},\bold{y}})}{\partial u^{l+p-a}\partial \bold{v}^{\bold{j}+\mathbf{q}-\mathbf{b}}}\left(u-\dfrac{k}{w}\right)^{l-c}
          \left(\mathbf{v}-\dfrac{\mathbf{m}}{w}\right)^{\mathbf{j}-\mathbf{d}} \  du \ d\bold{v}\\
            && =w^{a+|\mathbf{b}|}\sum_{l+|\bold{j}|= n-(p+|\mathbf{q}|-a-|\mathbf{b}|)+1}\dfrac{(-1)^{n-(p+|\mathbf{q}|-a-|\mathbf{b}|)}}{l! \bold{j}!}
           \sum_{c=0}^{l} \sum_{\mathbf{d}=0}^{\mathbf{j}} \binom{l}{c}\binom{\mathbf{j}}{\mathbf{d}}\displaystyle\sum_{k\in \mathbb{Z}}\sum_{\mathbf{m} \in\mathbb{Z}^{d}}  \\
           &&\times \phi^{(a)}(wx-k)\prod_{i=1}^{d}\psi_{i}^{(b_{i})}(wy_{i}-m_{i}) \left(\dfrac{k}{w}-x\right)^{c}\left(\dfrac{\mathbf{m}}{w}-\mathbf{y}\right)^{\mathbf{d}} w^{d+1}
         \\
          &&\times \int_{I_{k}^{w}}\int_{I_{\mathbf{m}}^{w}}\dfrac{\partial^{l+|\bold{j} |+p+|\mathbf{q}|-a-|\mathbf{b}|}f(\zeta_{u,x},\zeta_{\bold{v},\bold{y}})}{\partial u^{l+p-a}\partial \bold{v}^{\bold{j}+\mathbf{q}-\mathbf{b}}} \left(u-\dfrac{k}{w}\right)^{l-c}\left(\mathbf{v}-\dfrac{\mathbf{m}}{w}\right)^{\mathbf{j}-\mathbf{d}} \  du \ d\bold{v}.
         \end{eqnarray*}
        We now estimate 
        \begin{eqnarray*}
          && w^{n-(p+|\mathbf{q}|)+1}(w^{a+|\mathbf{b}|}K_{n-(p+|\mathbf{q}|-a-|\mathbf{b}|),w}^{\phi^{(a)},\psi^{(\mathbf{b})}} \partial^{p+\mathbf{q}-a-\mathbf{b}}f(x, \mathbf{y}) )-\displaystyle \sum_{l + |\mathbf{j}| = n-(p+|\mathbf{q}|-a-|\mathbf{b}|)+1} \\
    &&\times \dfrac{(-1)^{ n-(p+|\mathbf{q}|-a-|\mathbf{b}|)}}{l! \mathbf{j}!}\sum_{c=0}^{l} \sum_{\mathbf{d}=0}^{\mathbf{j}} \binom{l}{c}\binom{\mathbf{j}}{\mathbf{d}}\dfrac{m_{c}(\phi^{(a)})}{l-c+1}\prod_{i=1}^{d} \dfrac{m_{d_{i}}(\psi_{i}^{(b_{i})})}{j_{i}-{d_{i}+1}} \partial^{l + \mathbf{j}+p+\mathbf{q}-a-\mathbf{b}} f(x, \mathbf{y})\\
   &&=\sum_{l+|\bold{j}|= n-(p+|\mathbf{q}|-a-|\mathbf{b}|)+1}\dfrac{(-1)^{ n-(p+|\mathbf{q}|-a-|\mathbf{b}|)}}{l! \bold{j}!}
           \sum_{c=0}^{l} \sum_{\mathbf{d}=0}^{\mathbf{j}} \binom{l}{c}\binom{\mathbf{j}}{\mathbf{d}}\sum_{k \in \mathbb{Z}} \sum_{\mathbf{m} \in\mathbb{Z}^d}  \\
    &&\times \phi^{(a)}(wx - k)\prod_{i=1}^{d} \psi_{i}^{(b_{i})}(wy_{i}-m_{i})w^{d+1}\left(k-wx\right)^{c}\left(\mathbf{m}-w\mathbf{y}\right)^{\mathbf{d}}\int_{ I_{k}^{w}} \int_{ I_{\mathbf{m}}^{w}}\\
    &&\times\left(\partial^{l + \mathbf{j}+p+\mathbf{q}-a-\mathbf{b}} f(\zeta_{u,x}, \zeta_{\mathbf{v},\mathbf{y}})-\partial^{l + \mathbf{j}+p+\mathbf{q}-a-\mathbf{b}} f(x, \mathbf{y})\right) \left(wu-k\right)^{l-c} \left(w\mathbf{v}-\mathbf{m}\right)^{\mathbf{j}-\mathbf{d}}\, du \, d\mathbf{v}.
     \end{eqnarray*}
     Using the continuity of $\partial^{l + \mathbf{j}+p+\mathbf{q}-a-\mathbf{b}} f$ and moment conditions, we obtain
     \myeq{
         &\lim_{w\to\infty} w^{n-(p+|\mathbf{q}|)+1}(w^{a+|\mathbf{b}|}K_{n-(p+|\mathbf{q}|-a-|\mathbf{b}|),w}^{\phi^{(a)},\psi^{(\mathbf{b})}} \partial^{p+\mathbf{q}-a-\mathbf{b}}f(x, \mathbf{y}) )  \\ 
    &=\displaystyle \sum_{l + |\mathbf{j}| = n-(p+|\mathbf{q}|-a-|\mathbf{b}|)+1} \dfrac{(-1)^{n-(p+|\mathbf{q}|-a-|\mathbf{b}|)}}{l! \mathbf{j}!} \sum_{c=0}^{l} \sum_{\mathbf{d}=0}^{\mathbf{j}} \binom{l}{c}\binom{\mathbf{j}}{\mathbf{d}}\dfrac{m_{c}(\phi^{(a)})}{l-c+1}  \\ 
    &\times\prod_{i=1}^{d} \dfrac{m_{d_{i}}(\psi_{i}^{(b_{i})})}{j_{i}-{d_{i}+1}} \partial^{l + \mathbf{j}+p+\mathbf{q}-a-\mathbf{b}} f(x, \mathbf{y}).
     }
    Thus, combining the estimates (\ref{eq:3}), (\ref{eq:10}) and (\ref{eq:11}), we get
    \begin{eqnarray*}
          &&  \displaystyle\lim_{w\to \infty} w^{n-(p+|\mathbf{q}|)+1} (\partial^{p+\mathbf{q}}K_{n,w}^{\phi,\psi}f(x,\mathbf{y})-\partial^{p+\mathbf{q}}f(x,\mathbf{y}))\\
    &&=\sum_{a=0}^{p} \sum_{\mathbf{b}=0}^{\mathbf{q}} \binom{p}{a}\binom{\mathbf{q}}{\mathbf{b}}\displaystyle \sum_{l + |\mathbf{j}| = n-(p+|\mathbf{q}|-a-|\mathbf{b}|)+1} \dfrac{(-1)^{ n-(p+|\mathbf{q}|-a-|\mathbf{b}|)}}{l! \mathbf{j}!} \\
    &&\times\sum_{c=0}^{l} \sum_{\mathbf{d}=0}^{\mathbf{j}} \binom{l}{c}\binom{\mathbf{j}}{\mathbf{d}}\dfrac{m_{c}(\phi^{(a)})}{l-c+1}\prod_{i=1}^{d} \dfrac{m_{d_{i}}(\psi_{i}^{(b_{i})})}{j_{i}-{d_{i}+1}} \partial^{l + \mathbf{j}+p+\mathbf{q}-a-\mathbf{b}} f(x, \mathbf{y}).
    \end{eqnarray*}
    Thus, we get the required result.
 \end{proof}

\section{Implementation Results}
In this section, we consider the cardinal $B$-spline kernel and we show implementation results for differentiable function by these kernels.
For $n\in \mathbb{N},$ The cardinal $B$-spline of degree $n$ on $\mathbb{R}$ is defined by 
$$B_{n}(x)=(B_{0}*B_{0}*...*B_{0})(x), \ \ (n \ \text{convolutions}),$$
where $$B_{0}(x)=
\begin{cases} 
      1, & \text{if} \ |x|\leq \frac{1}{2}\\
      0, & \text{if}\ |x|>\frac{1}{2}.

   \end{cases}$$
\begin{example}
  We consider the function $$f(x,y)=\dfrac{(1+x)y}{1+x^{2}}, \ \ (x,y) \in [-2,2]\times[-2,2]$$  
  and its graph is shown in the figure \ref{Fig:1}. 
\begin{figure}[H]
  \centering
  \includegraphics[width=0.5\textwidth]{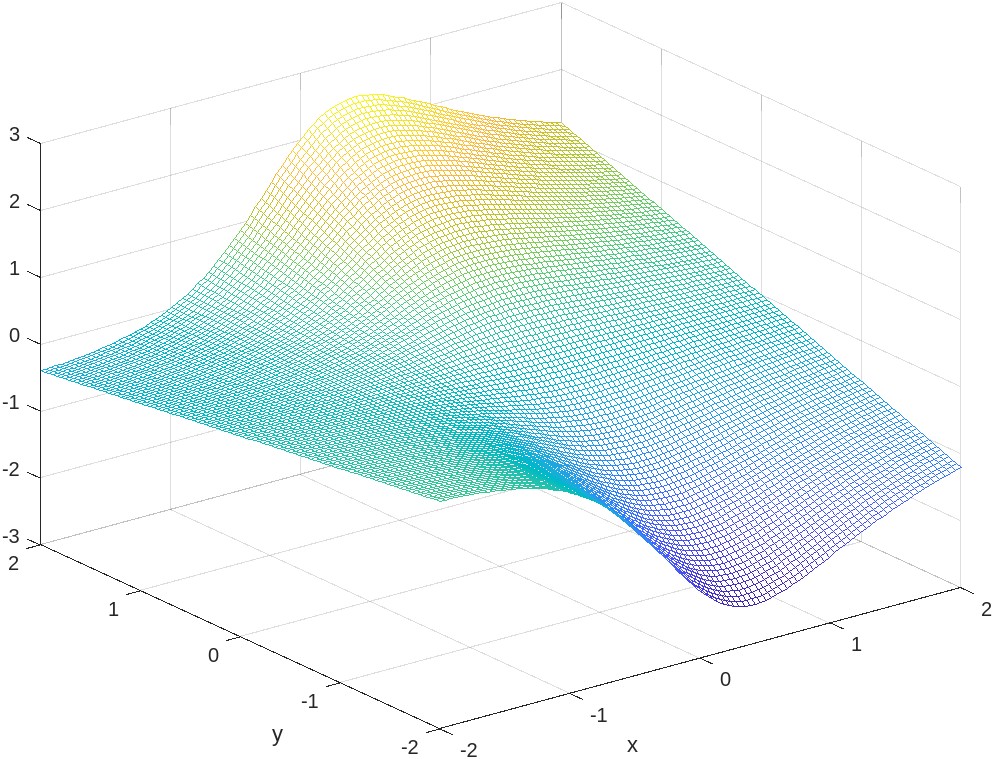}
  \caption{The function $f$}\label{Fig:1}
\end{figure}
 Now consider the following kernels:
$$\phi(x)=B_{2}(x), \ \psi(y)=B_{2}(y), \ \ x,y \in [-2,2].$$
Since $\phi$ and $\psi$ are compactly supported, they satisfy both conditions (i) and (ii) of the kernel.
 Plots of the approximations of f by $K_{n,w}^{\phi,\psi}f$ for values $n=0,1,2$  and for $w=7$  are shown in figure \ref{Fig:2}, figure \ref{Fig:3}, and figure \ref{Fig:4}, respectively.
 \begin{figure}[H]
  \centering
  \includegraphics[width=0.5\textwidth]{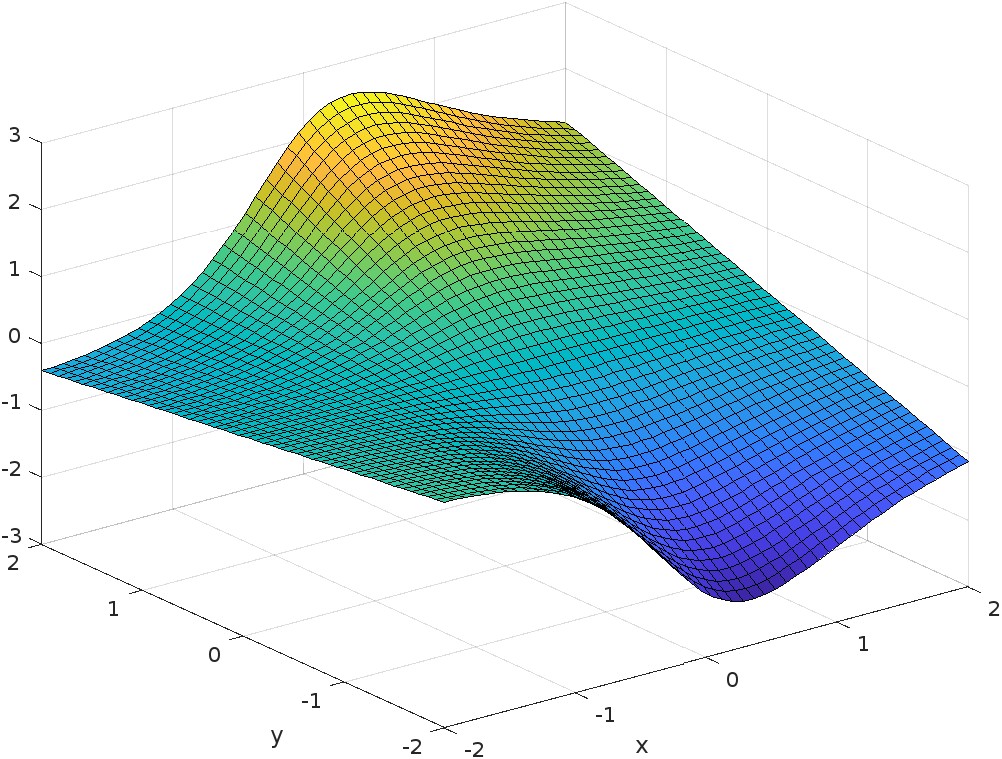}
  \caption{The approximation of $f$ by $K_{0,7}^{\phi,\psi}f$ }\label{Fig:2}
\end{figure}
\begin{figure}[H]
  \centering
  \includegraphics[width=0.5\textwidth]{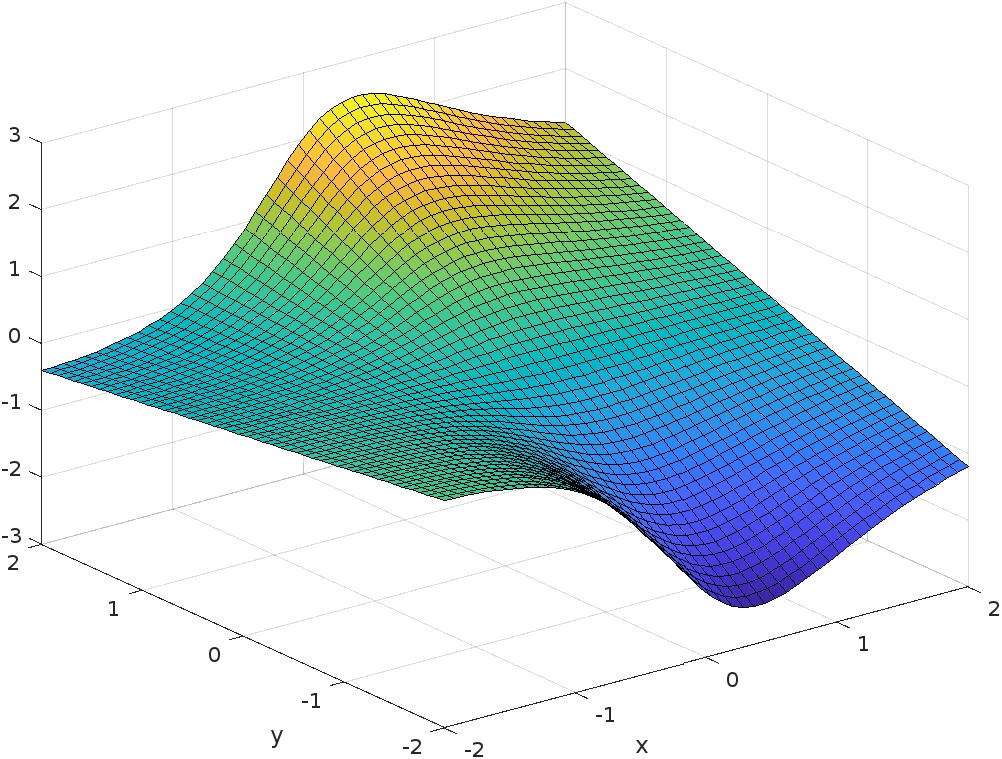}
  \caption{The approximation of $f$ by $K_{1,7}^{\phi,\psi}f$ }\label{Fig:3}
\end{figure}
\begin{figure}[H]
  \centering
  \includegraphics[width=0.5\textwidth]{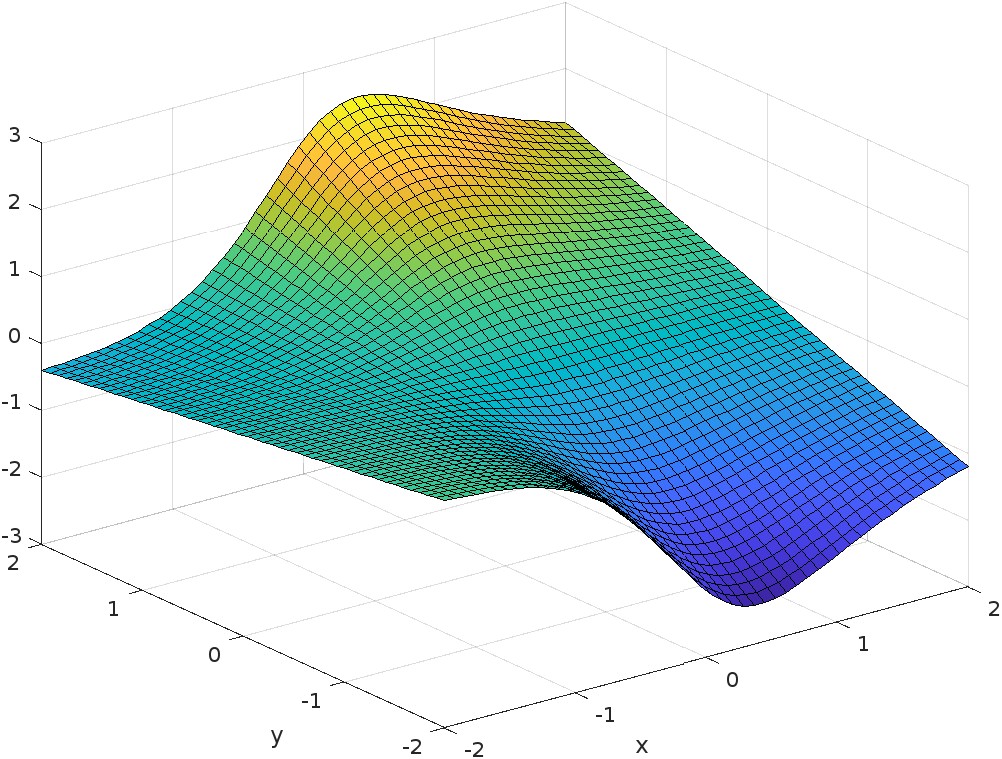}
  \caption{The approximation of $f$ by $K_{2,7}^{\phi,\psi}f$ }\label{Fig:4}
\end{figure}
   According to Theorem \ref{thm3}, increasing the order leads to a better approximation, which is reflected in the numerical errors.
 $$\|K_{0,7}^{\phi,\psi}f-f\|_{\infty}=0.2323, \ \|K_{1,7}^{\phi,\psi}f-f\|_{\infty}=0.0325, \ \|K_{2,7}^{\phi,\psi}f-f\|_{\infty}=0.0055.$$
 Figure \ref{Fig:5} illustrates the errors  $E_{n}(w)=\|K_{n,w}^{\phi,\psi}f-f\|_{\infty}$  as functions of $w$ for $n=1,2,$ along with the corresponding bounds $$T_{n}(w)=\dfrac{2^{n-1}}{w^{n+1}}\displaystyle \sum_{l+j= n+1}\dfrac{1}{l! j!}   \| \partial^{l+j }f\|_{\infty} (M_{l}(\phi)+M_{0}(\phi))   (M_{j}(\psi)+M_{0}(\psi)),$$ as stated in part (iii) of Theorem \ref{thm2}.
  Comparing the two figures, we observe that a faster convergence rate corresponds to a smaller value of $w$ required to achieve a given error.
 \begin{figure}[H]
  \centering
  \begin{subfigure}[b]{0.45\linewidth}
    \centering
    \includegraphics[width=\linewidth]{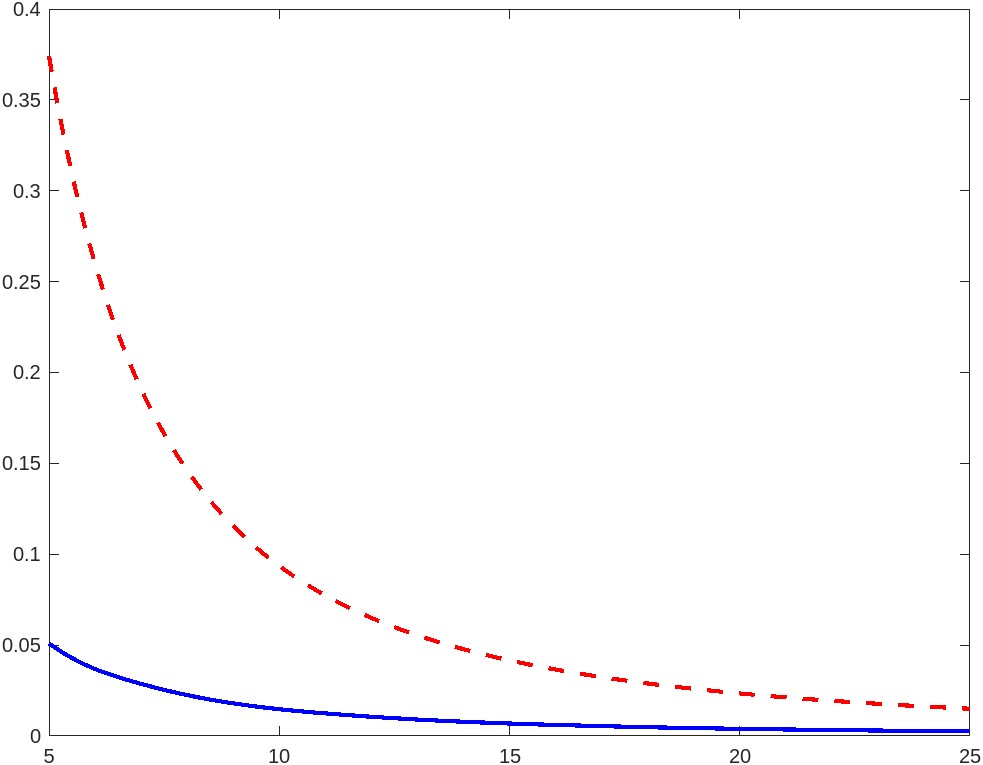}
    \caption{}
    \label{fig:sub1}
  \end{subfigure}
  \hfill
  \begin{subfigure}[b]{0.45\linewidth}
    \centering
    \includegraphics[width=\linewidth]{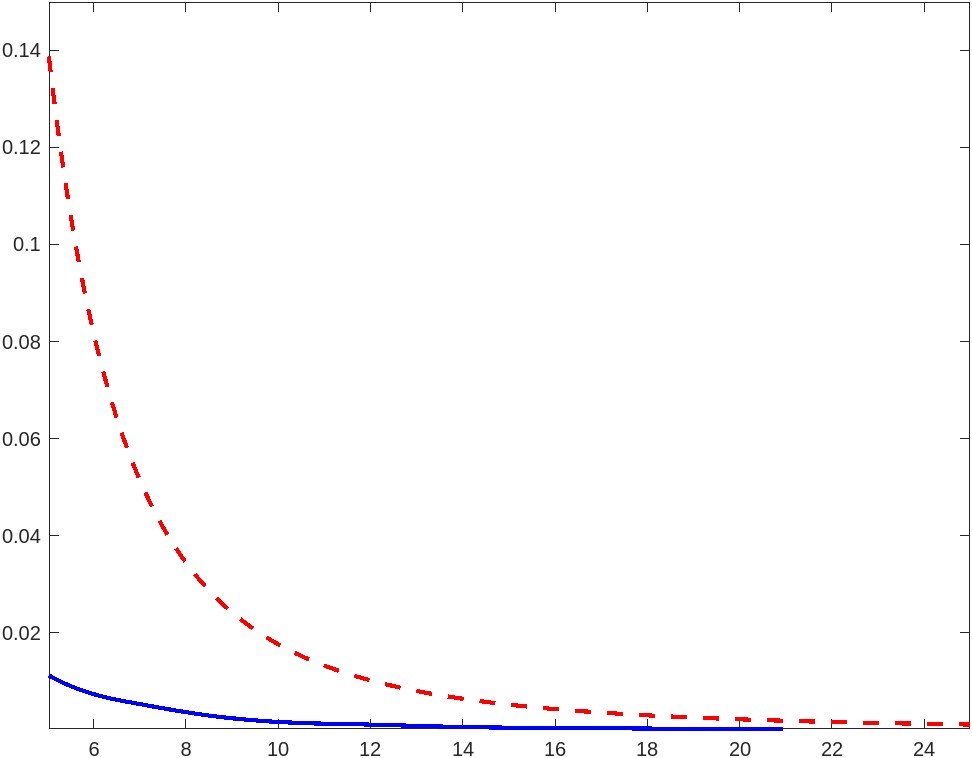}
    \caption{}
    \label{fig:sub2}
  \end{subfigure}
  \caption{\footnotesize(A) The solid line is a plot of the error $E_{1}(w)$ depending on $w$ and the dashed line is a plot of the bound $T_{1}(w)$ of part(iii) of Theorem \ref{thm2}. \vspace{0.1cm}\\ (B) The solid line is a plot of the error $E_{2}(w)$ depending on $w$ and the dashed line is a plot of the bound $T_{2}(w)$ of part (iii) of Theorem \ref{thm2}. }
  \label{Fig:5}
\end{figure}

\end{example}
\begin{example}
   Now we consider the function  $$f(x,y)=\sin(x)\cos(y), \ \ (x,y) \in [-4,4]\times[-4,4].$$ 
   We can view the result of the simultaneous convergence of Theorem \ref{thm4}. We will choose $p=1,\ q=1,\ n=3$ and set $\phi(x)=B_{4}(x),\ \psi(y)=B_{4}(y).$ Clearly, $\phi,  \psi$ satisfies the conditions of Theorem \ref{thm4}.
 We have  $$\dfrac{\partial^{2}f}{\partial x \partial y} =-\cos(x)\sin(y), \ \ (x,y) \in [-4,4]\times[-4,4].$$
   Its graph is shown in the figure \ref{Fig:6}.
   \begin{figure}[H]
  \centering
  \includegraphics[width=0.5\textwidth]{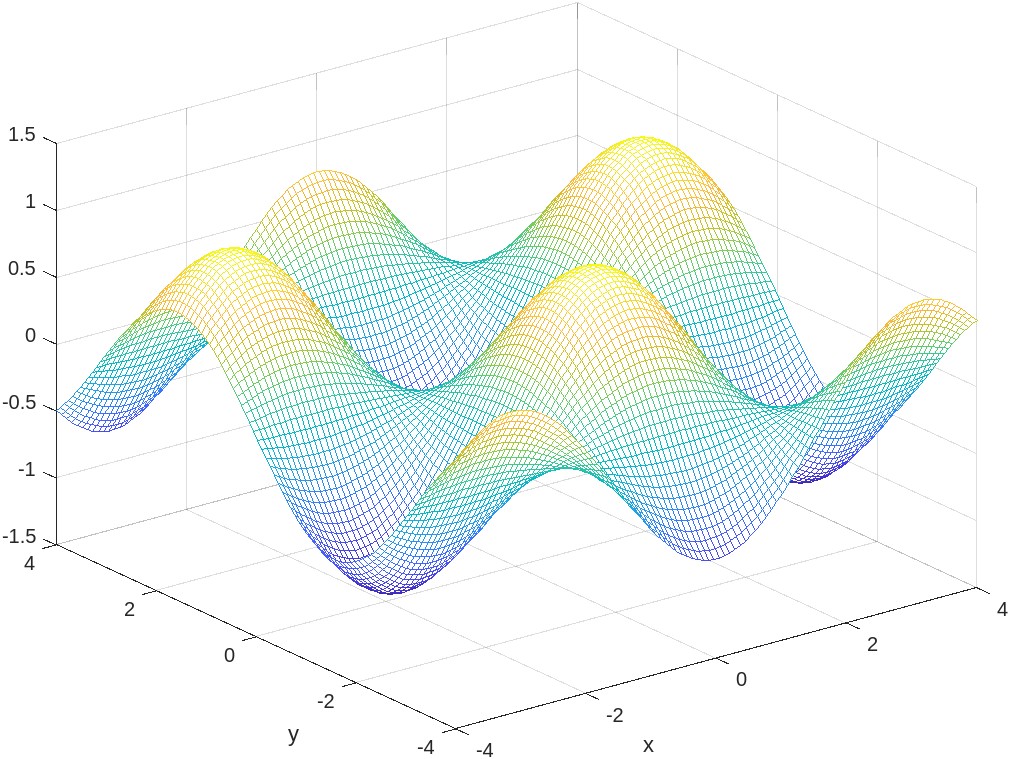}
  \caption{The function $f_{xy}$}\label{Fig:6}
\end{figure}
   Plots of the approximation  $\dfrac{\partial^{2}K_{3,w}^{\phi,\psi}f}{\partial x \partial y}$ of $f_{xy}$ for $w=3,7,12$ are shown in figure \ref{Fig:7}, figure \ref{Fig:8}, and figure \ref{Fig:9}, respectively.
   \begin{figure}[H]
  \centering
  \includegraphics[width=0.5\textwidth]{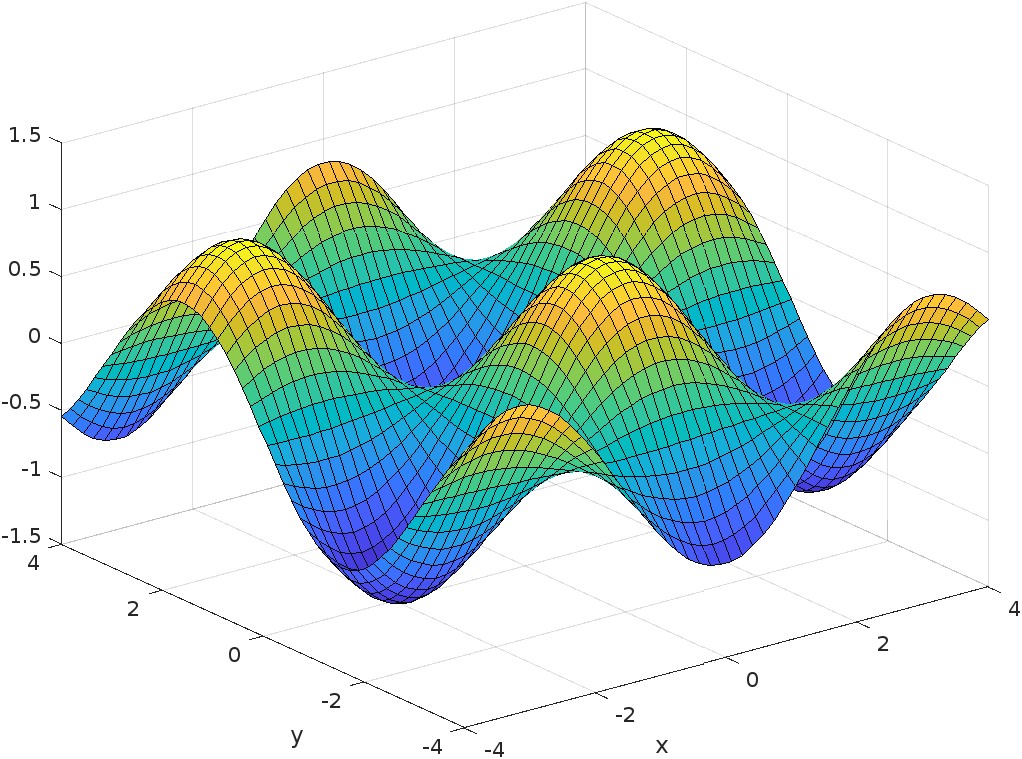}
  \caption{The approximation of $f_{xy}$ by $(K_{3,3}^{\phi,\psi}f)_{xy}$ }\label{Fig:7}
\end{figure}
\begin{figure}[H]
  \centering
  \includegraphics[width=0.5\textwidth]{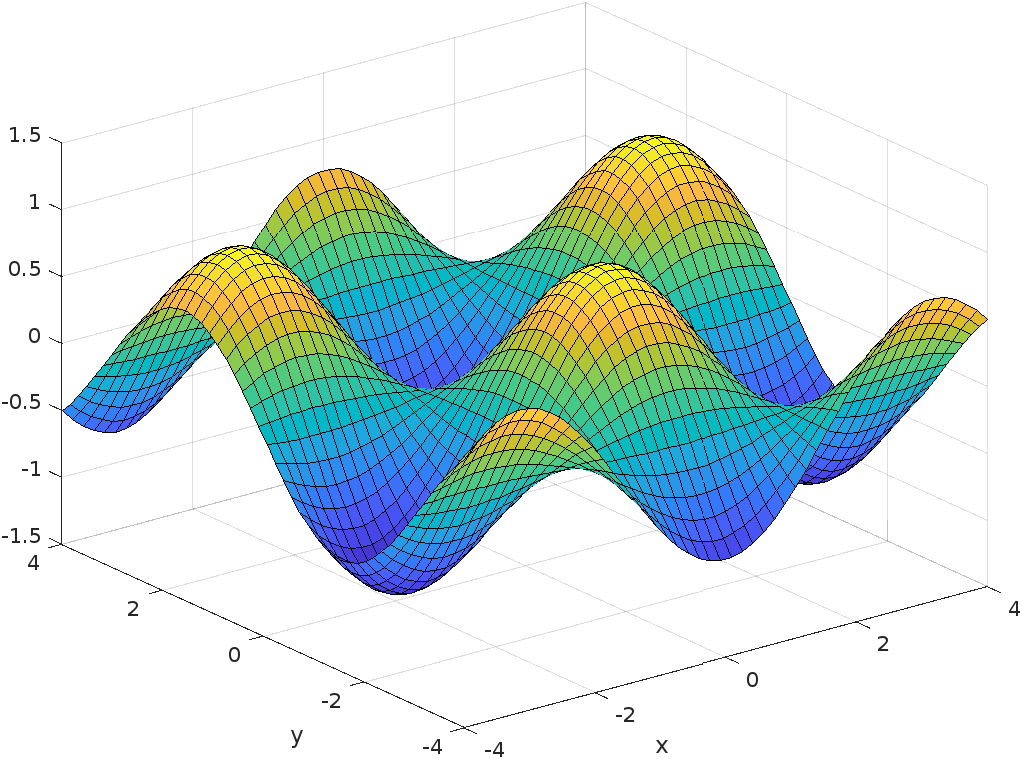}
  \caption{The approximation of $f_{xy}$ by $(K_{3,7}^{\phi,\psi}f)_{xy}$ }\label{Fig:8}
\end{figure}
\begin{figure}[H]
  \centering
  \includegraphics[width=0.5\textwidth]{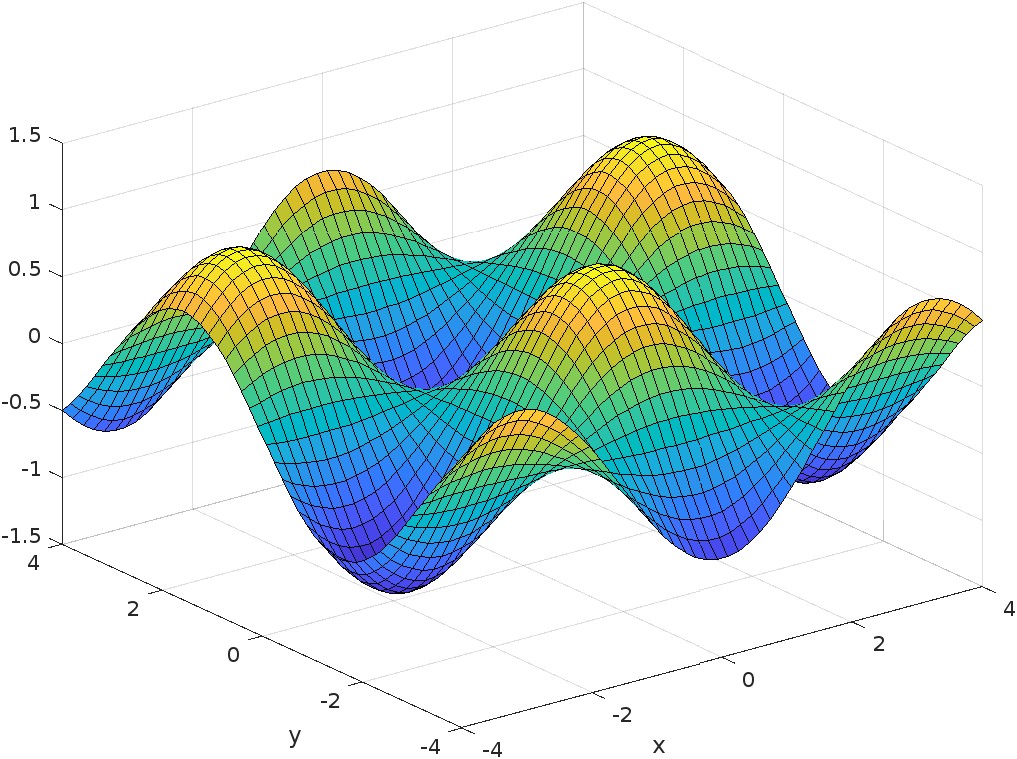}
  \caption{The approximation of $f_{xy}$ by $(K_{3,12}^{\phi,\psi}f)_{xy}$ }\label{Fig:9}
\end{figure}
As $w$ increases, approximation errors will reduce. This can be checked in Table \ref{table1}.
   \begin{table}[H]
\centering
\caption{}
\label{table1}
\begin{tabular}{|p{3cm}|p{4cm}|}
\hline
\centering$w$ &  \rule{0pt}{35pt}$\left\|\dfrac{\partial^{2}K_{3,w}^{\phi,\psi}f}{\partial x \partial y}-\dfrac{\partial^{2}f}{\partial x \partial y}\right\|_{\infty}$ \\[30pt]
\hline
 \centering 3& 0.0790  \\[5pt]
  \centering 7&  0.0151 \\[5pt]
  \centering 12 &  0.0052 \\[5pt]
\hline
\end{tabular}

\end{table}

\end{example}

\section{Conclusions} 
In this work, the approximation behavior of Hermite-type sampling Kantorovich operators in the context of mixed norm spaces are investigated. The direct approximation theorems have been proved, including including the uniform convergence theorem, the Voronovskaja-type asymptotic formula, and an estimate of error in the approximation in terms of the modulus of continuity in mixed norm settings. The rate of convergence of sampling Kantorovich operators in terms of the modulus of continuity has been studied. The simultaneous approximation results of these sampling Kantorovich operators, including the uniform approximation, the asymptotic formula, and the approximation error in terms of the modulus of continuity in mixed settings have been discussed. The realization of differentiable functions is also demonstrated using appropriate kernel functions.

\subsection*{Acknowledgments}

\noindent
Puja thanks the HTRA Fellowship, IIT Madras, for financial support to carry out her research work.
\noindent 
A. Sathish Kumar acknowledges Anusandhan National Research Foundation (ANRF), India, Research Grant: EEQ/2023/000257 for financial support. 

\subsection*{Data Availability:} Not Applicable. 

\subsection*{Declarations}
\noindent
Conflict of interest: The authors declare no competing interests.

\end{document}